\documentclass[10pt]{amsart}
\usepackage{amscd, amssymb}
\usepackage[matrix, arrow]{xy}

\begin{document}

\newcommand{\pnt}{\textup{pnt}}
\newcommand*{\KK}{\mathrm{KK}}
\newcommand*{\RKK}{\mathrm{RKK}}
\newcommand*{\K}{\mathrm{K}}
\newcommand*{\Ktop}{\mathrm{K}^{\mathrm{top}}}
\newcommand*{\KX}{\mathrm{KX}}
\newcommand*{\RRKK}{\mathscr{R}\KK}
\newcommand{\Dudelta}{\widehat{\Delta}}
\newcommand*{\Comp}{{\mathbb{K}}}
\newcommand*{\Bound}{{\mathbb{B}}}
\newcommand{\T}{\mathbb{T}}
\newcommand{\Q}{\mathbb{Q}}
\newcommand{\Z}{\mathbb{Z}}
\newcommand{\C}{\mathbb{C}}
\newcommand{\N}{\mathbb{N}}
\newcommand{\R}{\mathbb{R}}
\newcommand{\Hom}{\textup{Hom}}
\newcommand{\ip}[1]{\langle #1 \rangle}
\newcommand*{\defeq}{\mathrel{:=}}

\newcommand{\PD}{\mathrm{PD}}
\newcommand{\Isom}{\mathrm{Isom}}
\newcommand{\Eul}{\mathrm{Eul}}
\newcommand{\End}{\mathrm{End}}
\newcommand{\TT}{\mathbb{T}}
\newcommand{\Rep}{\textup{Rep}}
\newcommand{\Aut}{\textup{Aut}}
\newcommand{\trace}{\textup{trace}}
\newcommand{\Id}{\textup{Id}}
\def\E{\mathcal E}

\theoremstyle{plain}
 \newtheorem{theorem}{Theorem}[section]
\newtheorem{corollary}[theorem]{Corollary}
 \newtheorem{lemma}[theorem]{Lemma}
\newtheorem{lem}[theorem]{Lemma}
\newtheorem{prop}[theorem]{Proposition}
\newtheorem{proposition}[theorem]{Proposition} 
\newtheorem{lemdef}[theorem]{Lemma and Definition}
\theoremstyle{defn}
\newtheorem{defn}[theorem]{Definition}
\theoremstyle{definition}
\newtheorem{definition}[theorem]{Definition}
\newtheorem{defremark}[theorem]{\bf Definition and Remark}
\theoremstyle{remark}
\newtheorem{remark}[theorem]{Remark}
\newtheorem{assumption}[theorem]{Assumption}
\newtheorem{example}[theorem]{Example}
\newtheorem{examples}[theorem]{Examples}
\newtheorem{note}[theorem]{Note}
\numberwithin{equation}{section} \emergencystretch 25pt
\renewcommand{\theenumi}{\roman{enumi}}
\renewcommand{\labelenumi}{(\theenumi)}

\newcommand*{\abs}[1]{\lvert#1\rvert}
\newcommand{\card}{\mathrm{card}}
\def\supp{\operatorname{supp}}
\def\ind{\operatorname{ind}}

\newcommand{\Eulc}{\chi^{\Eul}}
\newcommand{\TTspec}{\TT\textup{-spec}}
\newcommand{\sheaf}{\mathcal{O}}
\newcommand{\Spec}{\textup{Spec}}
\newcommand{\spec}{\textup{-spec}}
\newcommand{\Laur}{\C[X, X^{-1}]}. 
\newcommand{\Laurs}{\C[X_1, X_1^{-1}, \ldots , X_n, X_n^{-1}]}
\newcommand{\rank}{\textup{rank}}
\newcommand{\ann}{\textup{ann}}
\newcommand{\Fixed}{\mathcal{C}}
\newcommand{\Tvert }{\mathrm{T}}
\newcommand{\Lef}{\textup{Lef}}
\newcommand{\Lindex}{\textup{Ind}^{L}}
\newcommand{\Ad}{\textup{Ad}}
\newcommand{\Div}{\textup{Div}}

\title[Localization techniques in circle-equivariant KK-theory]
  {Localization techniques in circle-equivariant Kasparov theory}

\author[Emerson]{Heath
 Emerson}
 \address{Department of Mathematics and Statistics,
 University of Victoria,
 PO BOX 3045 STN CSCVictoria,
 B.C.Canada}
 \email{hemerson@uvic.ca}
 
\thanks{The author was supported by a National Science and Engineering Council of Canada (NSERC) Discovery grant.}

\maketitle

\begin{abstract}
Let \(\TT\) be the circle and \(A\) be a \(\TT\)-C*-algebra. Then the \(\TT\)-equivariant 
\(\K\)-theory \(\K_*^\TT (A)\) is a module over the representation 
ring \(\Rep (\TT)\) of the circle. The latter is a Laurent polynomial ring. 
Using the support of the module as an invariant, 
and techniques of 
Atiyah, Bott and Segal, we deduce that there are
 examples of \(\TT\)-C*-algebras
\(A\) such that \(A\) and \(A\rtimes \TT\) are in the bootstrap 
category, but \(A\) is not \(\KK^\TT\)-equivalent to any 
commutative \(\TT\)-C*-algebra. We also assemble various results on 
\(\TT\)-equivariant \(\K\)-theory of smooth manifolds and deduce an 
equivariant version of the Lefschetz fixed-point formula for 
\(\TT\)-equivariant geometric correspondences. 

\end{abstract}

\def\mathcs{{\normalshape\text{C}}^{\displaystyle *}}

\section{Introduction}

This article has several purposes. The first is to show that
 many \(\TT\)-C*-algebras 
are not \(\KK^\TT\)-equivalent to any 
commutative \(\TT\)-C*-algebra, even though both they 
and their cross-products 
by \(\TT\) are in the boostrap category.  These 
examples include the Cuntz-Krieger algebras 
\(O_A\) with their usual circle actions.  To prove this statement 
we use a simple \(\K_*^\TT\)-theoretic 
obstruction to commutativity based on ideas of Atiyah, 
Bott and Segal. 

The phenomenom just described is in 
sharp contrast to the non-equivariant situation: every 
C*-algebra in the boostrap category is 
\(\KK\)-equivalent to a commutative one.

 Here and throughout this article, 
 \(\K_*^\TT(A):= \K_0^\TT(A) \oplus \K_1^\TT(A)\) denotes 
 equivariant \(\K\)-theory with \emph{complex coefficients}, 
 \emph{i.e.} is the integral \(\K\)-theory 
 tensored by \(\C\). In particular
  \(\Rep (\TT)= \KK^\TT(\C, \C)\cong \Laur\) 
 is the ring of Laurent polynomials with complex 
 coefficients. 
 
%Several related recent results of the author require study of
%the structure of \(\K_*^G(A)\) as a 
%module over \(\Rep (G)\). 
Study of equivariant \(\K\)-theory groups \(\K^*_G(X)\) 
as modules over \(\Rep (G)\) (\(G\) a compact group) 
 began with a series of papers by Atiyah, Bott and 
Segal, written in the 60's (see \cite{Atiyah-Bott:Moment},\cite{Atiyah-Segal:Index},
\cite{Atiyah-Segal:Euler},\cite{Segal}.)

 %The Universal Coefficient and 
 %Kunneth theorems require some simple homological algebra with 
 %\(\Z\). What is less obvious is the background role of 
 %algebraic geometry in such discussions. This becomes more 
 %clear when considering arbitrary \emph{connected} compact 
 %groups \(G\). 
 
A common strategy in these articles is first to prove results about 
the case \(G = \TT\), and then extend them 
to the case of connected groups \(G\) using Lie group theory (we will 
restrict entirely to \(\TT\) in this article.) The article 
\cite{Atiyah-Bott:Moment} treats equivariant
cohomology (for torus actions) and contains a lot of the essential ideas
used by us here, except that we work in equivariant 
\(\K\)-theory instead. Other good sources for 
equivariant \(\K\)-theory are the articles \cite{Segal} of Segal and 
Atiyah-Segal \cite{Atiyah-Segal:Index}. 
As we wish to reach a wider readership than only those who are 
familiar with these
articles, we have explained supports and 
localization rather carefully in this article.

Any module over \(\Rep (\TT)\cong \Laur\), and in particular, 
the module \(\K_*^\TT(A)\)
 for a \(\TT\)-C*-algebra \(A\), yields a sheaf of modules  
over \(\C^*\) defined by localizing the module to Zariski open sets.
 Such a sheaf has a support. 
The techniques of Atiyah and Segal are used in the first 
part of the article to check that, for any locally compact \(\TT\)-space 
\(X\),  the support of 
\(\K_*^\TT\bigl( C_0(X)\bigr) = \K^*_\TT (X)\) is 
  always either contained in the unit circle or is all of 
\(\C^*\).  But, as we observe,  
for a Cuntz-Krieger algebra \(O_A\) with its standard circle action, 
the support of the sheaf \(\K_*^\TT(O_A)\) is the set of nonzero eigenvalues 
of the 
\(0\)-\(1\)-valued matrix \(A\). 

Thus Cuntz-Krieger algebras have 
rather arbitrary algebraic integers as spectral points. In particular, they are not generally 
\(\KK^\TT\)-equivalent to commutative \(\TT\)-C*-algebras.

The second purpose of this paper is to strengthen several 
results on 
\(\TT\)-equivariant \(\K\)-theory of compact smooth 
manifolds due to
 Atiyah \emph{et al}, for example, proving that 
after a suitable localization, \(C(X)\) is \(\KK^\TT\)-equivalent 
to \(C(F)\) with \(F\subset X\) the stationary set, and to describe  
\(\TT\)-equivariant \(\K\)-theory for smooth manifolds in terms of 
various geometric data. This discussion is mainly for purpose 
of proving the\emph{ Lefschetz theorem} in \(\KK^\TT\). 

The 
equivariant Lefschetz theorem proved here generalizes the 
classical Lefschetz fixed-point 
formula. Recall that this formula 
equates a homological invariant of a smooth 
self-map \(f\colon X \to X\) with a geometric invariant of the map. 
Our equivariant Lefschetz formula 
takes into account a \(\TT\)-action for which the map is equivariant; 
moreover, it applies to more general morphisms in \(\KK^\TT(C(X), C(X))\) 
than just the ones induced from smooth maps: our techniques 
work just as well for geometric correspondences in the sense of 
\cite{Emerson-Meyer:Correspondences}.

Since \(\Laur\) is a principal ideal domain, 
any finitely generated \(\Laur\)-module \(M\) decomposes uniquely into a 
torsion module and a free module \(\cong \Laur^n\).  Any module 
self-map of \(M\) thus has a \(\Laur\)-valued trace by compressing  
it to the free part of \(M\). In particular, this applies to any 
element \(f\in \KK^\TT_*(C(X), C(X))\) where \(X\) is a \(\TT\)-space, 
for \(f\) acts by a module map on \(\K^*_\TT(X)\). We denote by 
\(\trace_{\Laur}(f_*)\in \Laur\) the \emph{graded} module trace of 
\(f_*\) in this sense.

In the case of a morphism \(f\) represented by a smooth \(\TT\)-equviariant 
geometric correspondence in the sense of \cite{Emerson-Meyer:Correspondences}, 
the Lefschetz theorem identifies the homological invariant 
\(\trace_{\Laur}(f_*)\) with 
the Atiyah-Singer \(\TT\)-index of a certain geometrically defined 
coincidence cycle constructed out of the correspondence: that is, we prove that
\begin{equation}
\label{intro:eq:lef_theorem}
 \trace_{\Laur}(f_*) = \ind_\TT \bigl( \Lef (f));
 \end{equation}
 where \(\Lef (f)\) is the class in \(\KK^\T\) of
  a certain \(\TT\)-equivariant Baum-Douglas 
 cycle for \(X\), depending geometrically on the correspondence 
 representing \(f\) and \(\ind_\TT\) is the Atiyah-Singer \(\TT\)-index.
 
  In particular the right hand side is 
defined purely 
in terms of equivariant correspondences and geometric 
intersections, and hence is a local, topological 
invariant of the correspondence. The left-hand-side is of course homological 
and global in nature. 

The equivariant Lefschetz theorem presented here is
 a special case of joint work with Ralf Meyer. See 
\cite{Emerson-Meyer:Lefschetz} for the more general version.

I would like to express my appreciation to Siegried Echterhoff and 
Ralf Meyer for their comments on the material here. The 
material in this note is related to joint work with both of them 
(independently.) I would also like to thank Nigel Higson 
for drawing my attention to the beautiful paper
\cite{Atiyah-Bott:Moment} of Atiyah and Bott on 
localization in equivariant cohomology.

Finally, the reader interested in further information on equivariant 
\(\K\)-theory for compact group actions should see the important 
source \cite{Rosenberg-Schochet:UCT}, which deals extensively with the 
Universal Coefficient and K\"unneth theorems in the integral version of 
\(\KK^\TT\),
 and more generally, for Hodgkin groups. When one works integrally,
 the representation ring \(\Rep (\TT)\) becomes \(\Z[X, X^{-1}]\) which is 
 no longer a principal ideal domain; this complicates some statements 
 considerably.

\section{The \(\TT\)-spectrum of spaces}

In the following, the reader should consider all \(\TT\)-equivariant 
\(\K\)-theory groups, \emph{e.g.} \(\K^*_\TT (X)\) for a 
\(\TT\)-space \(X\), or \(\K_*^\TT (A)\) for a \(\TT\)-C*-algebra 
\(A\), as having complex coefficients. Thus, \(\K_*^\TT(A)\) denotes 
the usual integral equivariant \(\K\)-theory of \(A\) 
tensored by the complex numbers.

Similarly, the symbol \(\Rep (\TT)\) means the 
usual representation ring of the circle, tensored with the 
complex numbers, or, more conveniently for us, 
the ring \(\Laur\) of Laurent polynomials in one variable, 
and complex coefficients. The isomorphism 
\(\Rep (\TT) \to \Laur\) is the character map. 

This note makes crucial use of the fact that for any 
\(\TT\)-C*-algebra \(A\), the \(\TT\)-equivariant \(\K\)-theory
 \(\K_*^\TT (A)\) is a module over \(\Rep (\TT)\cong \Laur\). 
 For unital, commutative
  \(\TT\)-C*-algebras this is rather clear, since in this 
  case \(\K_*^\TT (A)\) is a ring and the unital inclusion 
  \(\C \to A\) maps \(\Rep (\TT)\) to a subring of \(\K_*^\TT (A)\). 
  This induces the module structure. It is not hard to convince 
  oneself that if even if \(A\) is not unital, and hence no ring 
  embedding exists, the module structure 
  still makes sense. 
 
  In the general case, we may point to the 
  \emph{external product} in equivariant Kasparov theory 
  as a formal definition of the module structure: to translate to Kasparov 
  language, \(\K_*^\TT (A) = \KK^\TT_*(\C, A)\) and 
  \(\Rep (\TT) = \KK^\TT (\C, \C)\) (tensored by the complex numbers.)  
  So Kasparov external product gives grading-preserving 
  maps 
  \[ \KK^\TT_*(\C, A) \times \KK^\TT(\C, \C) \to \KK^\TT_*(\C, A)\]
  \[ \KK^\TT_*(\C, \C) \times \KK^\TT(\C, A) \to \KK^\TT_*(\C, A)\]
  These maps agree: external product is commutative. 
  
  More generally, \(\KK^\TT_*(A,B)\) is a graded \(\Rep (\TT)\)-module 
  for any \(A,B\).

  For commutative \(A\), \emph{i.e.} for \(\TT\)-spaces, 
  the module structure of \(\K^*_\TT (X)\) over 
  \(\Rep (\TT)\) has been quite extensively studied by 
  Atiyah and Segal in 
  [1] and [2], and also by Atiyah and Bott in the context of 
  equivariant cohomology in \cite{Atiyah-Bott:Moment}. 
  
The following definition applies to arbitrary \(\Laur\)-modules, 
and indeed, to modules over more general polynomial rings. 

\begin{definition}
\label{def:support}
Let \(M\) be a module over the ring \(\Rep (\TT) \cong \Laur\). Its
\emph{annihilator} \(\textup{ann}(M)\) is the ideal 
\(\{ f \in \Laur \mid fM = 0\}\).  
The 
\emph{support} of \(M\) is defined by 
\[ \supp (M) \defeq \bigcap_{f\in \textup{ann} (M)} Z_f\]
where \(Z_f\subset \C^*\) is the zero set of \(f\). 

\end{definition}

Thus a point 
\(z\) is not in the support of \(M\) if and only if there is a polynomial 
\(f\) such that \(f (z) \not= 0\) but \(fM = 0\). In particular, this can hold 
only if \(M\) has module torsion. Since a free module has 
no torsion, the support of a free module like \(\Laur\) itself, is \(\C^*\).

Under embeddings \(M_1\to M_2\) of 
\(\Laur\)-modules, supports can only increase as 
\(\ann (M_2)\subset \ann (M_1)\) in this situation, which 
implies \( \supp (M_1) \subset \supp (M_2)\). In particular, 
\(\supp (M) = \C^*\) as soon as \(M\) contains a free 
submodule. If on the other hand 
one has a surjection \(M_1\to M_2\), then \(\ann(M_1) \subset \ann(M_2)\) 
so that \(\supp(M_2) \subset \supp (M_1)\) results.

The ring \(\Laur\) is a principal ideal domain, \emph{i.e.}
 any ideal 
is generated 
 by a single polynomial \(f\). This polynomial is unique up to 
 multiplication by an invertible in \(\Laur\), \emph{i.e.} \(f\) can be replaced by 
 \(fX^n\) for any integer \(n\), and in particular \(f\) may always be 
 taken to be a polynomial. Furthermore, any finitely generated module 
 over a principal ideal domain decomposes uniquely into a 
 direct sum of a free module and a 
 torsion module. The torsion sub-module is by definition 
 \(\{m \in M \mid fm = 0 \; 
 \textup{for some } \; f\not= 0 \;  \;\textup{in} \; \Laur\}.\) 
 
 A finitely generated torsion module has a nonzero annihilator ideal 
 because the annihilator ideal
  is the intersection of the annihilator ideals of the generators, this 
 is an intersection of finitely many nonzero ideals and hence is nonzero. 
 If the annihilator of the torsion module is generated by \(f\), then 
 the support of the torsion module is the zero set \(Z_f\)  of \(f\) in \(\C^*\), and 
 in particular is a finite set of points of \(\C^*\). 
 If the module is not finitely generated, it may be torsion, but have a 
 zero annihilator ideal, however. In this case, the support will be 
 \(\C^*\) (see below for an example.) 

 If a module has finite dimension as a vector space over \(\C\) then 
 of course it is torsion and finitely generated and the above 
  discussion applies.

For any \(\Laur\)-module, ring multiplication by 
 \(X\in \Laur\) is an invertible, complex linear
 operator on the module, viewed just as a complex vector space. 
 If \(M\) is torsion with nonzero annihilator ideal, then 
 \emph{the support is 
 the set of eigenvalues of \(X\) and the generator \(f\) of the 
 annihilator ideal is the minimal polynomial of \(X\)}. Indeed, factor
  \( f(X) = (X-\lambda_1)^{k_1} \cdots (X-\lambda_n)^{k_n}\). Each
  \(\lambda_i\) must be an eigenvalue of \(X\) since
  \(\prod_{j\not= i} (X- \lambda_j )^{k_j} (X-\lambda_i)^{k_i-1}\) maps \(M\) into the 
  kernel of \(X-\lambda_i \). If the kernel of \(X- \lambda_i \)
   is zero, we would have 
  a polynomial of smaller degree annihilating \(M\), false.  So 
  the kernel is nonzero.   Furthermore, as \(f(X)  = 0\) on \(M\), 
 \(0 = f(X)v =  f(\lambda)v\) if \(v\) is any eigenvector of \(X\) with 
 eigenvalue \(\lambda\). Hence any eigenvalue of \(X\) is a root of \(f\).

 \begin{remark}
 \label{rem:finite_generation}
 Finite generation is guaranteed for the \(\Laur\)-module 
 \(\K^*_\TT (X)\) whenever 
 \(X\) is a smooth, compact manifold and \(\TT\) acts smoothly
 (see \cite{Segal}) or the discussion in Section 4 of this paper. 
\end{remark} 
 
%  Let \(V = \oplus_{1}^\infty V_n\) be a sum 
%  of finite-dimensional complex vector spaces 
 %\(V_n\) and \(T_n\) a linear automorphism of 
 %\(V_n\) with eigenvalues the group \(\Omega_n \subset \TT \) 
 %of complex \(n\)th roots of unity
 % and let \(X\) act on the sum with \(X\) acting 
 %by \(T_n\) on \(V_n\). 
 %Then 
 %\(V\) is torsion, \(X\) has eigenvalues \(\bigcup_n \Omega_n\), 
  %but the annihilator ideal of the 
 %module is zero and the support is \(\C^*\), so the 
 %support of \(V\) differs from the set of eigenvalues of 
 %\(X\) in this case, which it must, of course, since the 
 %support is always Zariski closed. See 
 %Example 
 %\ref{ex:points} for more discussion.  

\begin{definition}
\label{def:tspec}
Let \(A\) be a \(\TT\)-C*-algebra. The 
\emph{\(\TT\)-spectrum of \(A\)} is defined to be 
the support of 
\(\K_*^\TT(A)\) as an \(\Laur\)-module. 

\end{definition}

In the commutative case, we refer to the \(\TT\)-spectrum of 
the corresponding space.  

\begin{remark}
\label{rem:graded_spectrum}
The definition of spectrum in terms of the module 
\(\K_*^\TT(A) \defeq \K^\TT_0(A)\oplus \K^\TT_1(A)\) 
given above does not take into 
account the grading on \(\TT\)-equivariant \(\K\)-theory. 
A more natural invariant, in some ways, would take this 
into account, but we do not do this here because it is not 
necessary for our purposes. 
Note also that \(\K_*^\TT(A) \cong \K_0^\TT\bigl( C(S^1)\otimes A\bigr)\)
where the \(\TT\)-action on the circle \(S^1\) is trivial, which means 
that in computing module structures we can deal exclusively with 
\(\TT\)-equivariant vector bundles. 

\end{remark}

In the case of the trivial \(\TT\)-action on 
a point,  \(\K^*_\TT (\cdot) = \Rep (\TT) \) and 
the module structure over \(\Rep (\TT)\) is by ring multiplication. 
Hence the annihilator ideal is zero and \(\TTspec (\cdot) = \C^*\). 

If \(A = C([0,\infty))\) with trivial \(\TT\)-action, then 
\(\K_*^\TT (A) = 0\) and hence \(\TTspec (A) = \emptyset\) in 
this case.

Note that, although evaluation of Laurient polynomials at 
any \(z\in \C^*\) yields a \(\Laur\)-module \(M\) such that 
\(\supp (M) = \{z\}\), if this module is to arise from 
an equivariant \(\K\)-theory module, then \(z\) must be an 
algebraic integer, at least if the module is finite dimensional 
over \(\C\). 

\begin{proposition}
If \(\K_*^\TT (A)\) is finite-dimensional over \(\C\), then 
the \(\TT\)-spectrum of \(A\) is a 
finite set of algebraic integers in \(\C^*\). 
\end{proposition}

\begin{proof}
The spectrum in this case is the spectrum  
of \(X\) acting on \(\K_*^\TT (A)\). 
But \(X\) comes from an endomorphism of the 
underlying \(\TT\)-equivariant \(\K\)-theory with 
\emph{integer coefficients} and therefore is represented
in some basis for \(\K_*^\TT(A)\) by a matrix with 
integer coefficients, and \(\TTspec (A)\) is its set of eigenvalues, 
so they are algebraic integers. 
\end{proof}

\begin{theorem}
\label{thm:commutative_case}
If \(A = C_0(X)\) is \emph{any} commutative 
\(\TT\)-C*-algebra, then either \(\TTspec (A) = \C^*\) 
or \(\TTspec (A) \subset \TT\). In the latter case, 
the spectrum is finite and each point of it is 
an \(n\)th root of unity where \(n\) is the order of some 
(finite) isotropy group of the action.  

If \(X\) is compact, then 
\(\TTspec (X) = \C^*\) if and
only if \(X\) has a stationary point. 

\end{theorem}

The proof will occupy the rest of this section. We start by 
discussing stationary points. Suppose \(X\)  
has such a point. Then there is a 
\(\TT\)-map from the one-point \(\TT\)-space to \(X\); it 
induces a module map 
\(\K^*_\TT(X) \to \K^*_\TT (\cdot) = \Rep (\TT) \). If \
\(X\) is \emph{compact} this map is surjective 
because the map from \(X\) to a point is proper in this 
case and gives a splitting. Hence 
\(\C^* = \TTspec (\cdot) \subset \TTspec(X)\).

Thus, \(\TTspec (X) = \C^*\) if \(X\) has 
a stationary point and is compact. This is rather common; for example, 
by the Hopf theorem any smooth \(\TT\)-action on a 
smooth manifold of nonzero Euler characteristic 
has a stationary point. Hence having \(\TT\)-spectrum 
\(\C^*\) is rather generic for compact \(\TT\)-spaces. 

If \(X\) is not compact, it may have a stationary point 
without the spectrum being \(\C^*\); for example 
\([0,\infty)\) with trivial \(\TT\)-action has empty spectrum 
but many stationary points. The other implication also 
requires compactness in view of Example \ref{ex:points}
below, where the spectrum is \(\C^*\) but there is 
no stationary point.  

To get an example of a space with non stationary point but 
with spectrum \(\C^*\), observe first that 
for any collection \( (M_\lambda)_{\lambda \in \Lambda})\) of 
nonzero \(\Laur\)-modules, the annihilator of the direct sum 
\(M \defeq \bigoplus M_i\) is, essentially tautologically, the intersection 
\(\bigcap_i \ann(M_i)\) of the annihilators. 
But for the ring \(\Laur\), there can only be finitely 
many ideals containing a given nonzero ideal,
 for if the given one is generated by \((f)\) then any ideal 
 containing \((f)\) is generated by a divisor of \(f\). 
  Hence if there are infinitely many distinct ideals \( \ann(M_i)\), 
 then \(\bigcap_i \ann (M_i)\) would have to be the zero 
 ideal.
 
  This shows the following.

\begin{lemma}
\label{lem:adding}
If \(X\) is a \(\TT\)-space which is a disjoint union 
\(X = \bigsqcup_i X_i\) for a family of 
\(\TT\)-spaces \(X_i\). Then either the \(\TT\)-spectrum of 
\(X\) is \(\C^*\) or the sets \(\TTspec (X_i)\) 
are all finite, there are only finitely many of them,
and \(\TTspec (X) \) is their union. 
\end{lemma}

\begin{proof} 
\(\K^*_\TT (X) = \oplus_{i \in \Lambda} \K^*_\TT (X_i)\)
and the result follows from the preceding remarks. 
\end{proof}

\begin{example}
\label{ex:points}
Let \(\TT\) act on \(X_n \defeq \TT\) with \(t\cdot s 
\defeq t^ns\). Let \(\Omega_n \subset \TT\) denote the subgroup of 
\(n\)th complex root of unity. Then \(X_n \cong \TT/\Omega_n\) with 
\(\TT\) acting by translation on the quotient. Thus \(\K^*_\TT (X_n) \cong 
\Rep (\Omega_n)\), and the \(\Laur \cong \Rep (T)\) module structure 
is by restriction of representations, \emph{i.e.} by restrictions of 
polynomials to \(\Omega_n \subset \C^*\). The support is 
\(\Omega_n\), thus \(\TTspec (X_n) = \Omega_n\). 

Now let
 \(X = \TT\times \N\) with \(\TT\) acting as above
  in the \(n\)th copy of \(\TT\).  By Lemma \ref{lem:adding}, the 
\(\TT\)-spectrum of \(X\) is \(\C^*\), although there is 
no stationary point. 
\end{example}

The proof of Theorem \ref{thm:commutative_case} follows 
from the following two lemmas. 

\begin{lemma}
\label{lem:excision}
Let \(X\) be any \(\TT\)-space and
 \(Y \subset X\) be a closed \(\TT\)-invariant subspace of 
\(X\). Then 
\[ \TTspec (X) \subset \TTspec (Y) \, \cup \, \TTspec (X- Y).\]

\end{lemma}

\begin{proof}
Consider the \(6\)-term exact sequence of \(\TT\)-equivariant 
\(\K\)-theory groups associated to the exact sequence 
\[ 0 \to C_0(X- Y) \xrightarrow{i} C(X) \xrightarrow{r} C_0(Y) \to 0.\]
Let \(f\in \Laur\) annihilate \(\K^*_\TT (X-Y)\) and 
\(\K^*_\TT (Y)\). Then if \(a \in \K^0_\TT (X)\), 
\(0 = f\cdot r_*(a) = r_*(f\cdot a)\) implies \(f\cdot a = i_*(a')\) 
some \(a'\in \K^0_\TT(X-Y)\) and then \(f^2\cdot a = i_*(f \cdot a')= 0\)
so \(f^2\) annihilates \(\K^0_\TT (X)\). Similarly \(f^2\) annihilates 
\(\K^1_\TT (X)\).  Thus 
\[ f\in \textup{ann} \bigl( \K^*_\TT (X- Y)\bigr) \,  \cap \,  
\textup{ann}\bigl( \K^*_\TT (Y)\bigr) \Rightarrow f^2\in  \textup{ann}\bigl( 
\K^*_\TT (X)\bigr).\]
Hence \(\supp \bigl( \K^*_\TT (X)\bigr)\) is contained in  
\(Z_{f^2} = Z_f\) for any \(f\in \textup{ann} \bigl( \K^*_\TT (X- Y)\bigr) \,  \cap \,  
\textup{ann}\bigl( \K^*_\TT (Y)\bigr)\). The result now follows.
\end{proof}

\begin{lemma}
\label{lem:spec_for_slices}
If \(X \defeq \TT\times_H Y\) for some closed subgroup 
\(H \subset \TT\) and some \(H\)-space \(Y\), then 
\(\TTspec (X)\subset H\). In particular, if \(H\) is a 
proper subgroup, then the \(\TT\)-spectrum of \(X\) 
consists of a set of \(n\)th roots of unity, where 
\(n\) is the cardinality of \(H\). 
\end{lemma}

\begin{proof}
The \(\Rep (\TT)\)-module structure on \(\K^*_\TT (X) \cong 
\K^*_H (Y)\) factors through the restriction map 
\(\Rep (\TT) \to \Rep (H)\) and the \(\Rep (H)\)-module structure 
on \(\K^*_H(Y)\). If \(f\) is a polynomial which vanishes on 
\(H\subset \TT\) then it restricts to zero in \(\Rep (H)\) and hence acts by 
zero on \(\K^*_H(Y) \cong \K^*_\TT (X)\). Hence \(\TTspec (X) \subset H\) as 
claimed. 
\end{proof}

\begin{remark}\label{rem:slices}
We remind the reader of two easy and well-known 
facts about induced spaces. 
\begin{enumerate}
\item Induced spaces
 \(W = \TT\times_HY\) from a subgroup \(H \subset \TT\)  
 are characterised amoung
\(\TT\)-spaces as those admitting a \(\TT\)-map \(\varphi \colon  
W \to \TT/H\). We can recover \(Y\) from \(\varphi\) as the 
fibre over the identity coset in \(\TT/H\). 
\item We often call 
induced spaces \emph{slices}. Since we can always restrict a 
\(\TT\)-map to a \(\TT\)-invariant subspace, 
any \(\TT\)-invariant subspace of a slice is a slice too. 
\item A theorem of Palais (see \cite{Palais:Slices}) 
asserts that any \(\TT\)-space can be
covered by \emph{open} slices using stabilizer subgroups of the action. That is, 
 if \(X\) is any \(\TT\)-space and \(x\in X\), then 
 there exists an open subset 
\(U \subset X\) with \(x\in U\), and  a \(\TT\)-map  
\(\varphi \colon U \to \TT/H\) where \(H \defeq \TT_x\) is the 
stabilizer of \(x\).  (This result holds more generally for actions of 
\emph{Lie groups}.) 

\end{enumerate}
Note that if \(\varphi \colon U \to \TT/H\) is a slice with 
\(H = \TT_x\) for some \(x\in U\), then 
\( \TT_y \subset \TT_x\) for any \(y \in U\). 
\end{remark}

\begin{lemma}
\label{lem:lemmas}
Let \(X\) be a (locally compact) 
\(\TT\)-space. 
\begin{enumerate}
\item If \(X\) has no stationary points, then \(\K^*_\TT (X)\)
 is a torsion module and 
  \(\TTspec (Y)\) is a finite subset of \(\TT\) for every pre-compact 
  \(\TT\)-invariant subset \(Y \subset X\). Furthermore, 
  \(\TTspec (Y) \subset \cup_{y \in \overline{Y}} \TT_y\). 
\item If \(\TTspec (X)\) is finite, \(F\subset X\) is the stationary set, 
then \(\K^*_\TT (F) = 0\) and the \(\TT\)-equivariant *-homomorphism 
\(C_0(X- F)  \to C_0(X)\) determines an 
isomorphism \(\K^*_\TT (X- F) \cong \K^*_\TT (X)\) of 
\(\Laur\)-modules. 
\end{enumerate} 

\end{lemma}

\begin{remark}
The same arguments prove a stronger version of the 
second statement: that the \(\TT\)-equivariant *-homomorphism 
\( C_0(X- F) \to C_0(X)\) is invertible in 
\(\KK^\TT ( C_0(X- F), C_0(X))\). 
%See 
%Theorem \ref{thm:restriction_to_stationary}.

The condition of having finite \(\TT\)-spectrum thus implies that the 
stationary set \(F\) is \emph{homologically trivial}: that is, \(\K^*_\TT (F) = 0\). 
 Compare 
the ray \([0,\infty) \) with the trivial action.
\end{remark}

\begin{proof}
For the first statement, 
\(\K^*_\TT (X)\) is the inductive limit of the 
\(\K^*_\TT (Y)\), as \(Y \subset X\) ranges over the 
pre-compact \(\TT\)-invariant subsets of \(X\). Therefore, 
if we can prove that the annihilator ideal of 
\(\K^*_\TT (Y)\) is nonzero for every pre-compact 
\(\TT\)-invariant subset \(Y \subset X\), we will be 
done. This is equivalent to showing that \(\TTspec (Y)\) is 
finite for all such \(Y\).
 If \(Y \subset X\) is precompact, 
with closure \(\overline{Y}\), then we can cover 
\(\overline{Y}\) by finitely many \(\TT\)-slices 
\(\varphi \colon U_i \to \TT/H_i\) using 
stabilizer subgroups \(H_i\) of the action on \(\overline{Y}\).
 This gives a finite 
cover of \(Y\) itself by open slices (as in Remark
 \ref{rem:slices}, intersecting a slice with a 
 \(\TT\)-invariant subset always results in a slice.) Furthermore, 
 since the \(H_i\) are stabilizer groups of points in \(\overline{Y}\)
  and the action 
 has no stationary points, all \(H_i\) are finite subgroups of 
 \(\TT\). 

Now prove the result 
by induction on the minimal number of slices required to cover \(Y\), 
which we have just observed is finite.  
It can be covered by a single slice, then it is itself a slice, 
and the result follows from Lemma \ref{lem:spec_for_slices}. 
If the result is true for precompact subsets of 
\(X\) that can be covered by  
\(< n\) slices, and \(Y\) can be 
covered by \(n\) slices with domains 
\(U_1, \ldots , U_n\) and subgroups \(H_i\), 
then the closed \(\TT\)-invariant 
subspace \(Y - U_n\) of \(Y\) can be covered by 
\(n-1\) slices so by inductive hypothesis 
\(\TTspec(Y- U_n) \subset 
\bigsqcup_{x\in Y} \TT_x \subset \TT\) is finite.  
The result for 
\(Y\) now follows from 
Lemma \ref{lem:excision}.

For the second statement, consider the exact sequence of 
\(\TT\)-C*-algebras 
\[ 0 \to C_0(X- F) \rightarrow C_0(X) \rightarrow C_0(F) \rightarrow 0.\]
This induces an exact sequence of \(\K^*_\TT \)-groups. The restriction map 
\(\K^*_\TT (X) \to \K^*_\TT (F)\) must vanish, because we have assumed that 
\(X\) has finite 
spectrum, (\emph{i.e.} \(\K^*_\TT (X)\) is torsion)
whereas \(\K^*_\TT (F)\) is free. This implies that we have a pair of 
short exact 
sequences 
\[ 0 \rightarrow \K^{*+1}_\TT (F)
%\cong \K^{*+1}(F)\otimes \Laur 
\rightarrow \K^*_\TT (X- F) \to 
\K^*_\TT (X) \rightarrow 0\]
for \(* = 0,1\). But \(X- F\)  has no stationary points, 
so from the first part of this Lemma,
 \(\K^*_\TT (X- F)\) is torsion. 
 Now a free \(\Laur\)-module 
\(\K^{*+1}_\TT (F)\) which injects into a torsion module 
\(\K^*_\TT (X- F)\) can only be the zero module. 
Hence \(\K^*_\TT (F) = 0\), \(*=0,1\) and 
\(C_0(X- F) \to C(X)\) induces an isomorphism 
on \(\K^*_\TT\)-theory. 
\end{proof}

\begin{proof} (\emph{Of Theorem \ref{thm:commutative_case}}.) 
Assume that \(\TTspec (X) \not= \C^*\). 

Then since \(X\) has finite \(\TT\)-spectrum, \(\K^*_\TT (X) \cong 
\K^*_\TT (X- F)\) as \(\Laur\)-modules, 
where \(F\subset X\) is the 
stationary set, by Lemma \ref{lem:lemmas}. In particular, 
\(\TTspec (X) = \TTspec (X- F)\) so by replacing 
\(X\) by \(X- F\) we may assume that \(X\) itself has 
no stationary points. 

Now by the preliminary discussion following Definition 
\ref{def:support}, since \(\K^*_\TT(X)\) has nonzero annihilator 
ideal, the support is the set of eigenvalues of \(X\) acting on 
\(\K^*_\TT(X)\). Suppose \(\lambda\) is an eigenvalue, \(v\in \K^*_\TT (X)\) 
an eigenvector for \(\lambda\). Since \(\K^*_\TT (X)\) is the inductive 
limit of the \(\K^*_\TT (Y)\) as \(Y\subset X\) ranges over the 
precompact \(\TT\)-invariant subsets of \(X\), 
there exists precompact \(Y\) and \(w\in \K^*_\TT (Y)\) mapping to 
\(v\). By Lemma \ref{lem:lemmas}, since there are no stationary 
points, \(\K^*_\TT (Y)\) has a finite annihilator ideal, say generated
by \(g\in \Laur\), and moreover, the support of \(\K^*_\TT(Y)\) is 
contained in the unit circle. Since \(gw= 0\), \(gv = 0\), and as 
\(gv = g(\lambda) v\), 
\(\lambda\) is a root of \(g\), whence \(\lambda\) is contained in the unit circle
as claimed. 

This also proves that \(\lambda\) is an \(n\)th root of unity where 
\(n\) is the cardinality of some isotropy group of the action (on \(Y\)).  
 
 \end{proof}
 
%The following is immediate. 

%\begin{corollary}
%If \(X\) is a compact, free \(\TT\)-space then \(\TTspec(X) \subset \{1\}\). 
%\end{corollary}

\section{The \(\TT\)-spectra of C*-algebras}

Let \(B\) be a C*-algebra equipped with an automorphism
\(\sigma\). Then \(A\defeq B\ltimes \Z\) is a 
\(\TT\)-C*-algebra using the dual action 
\[ z (\sum_{n\in \Z} b_n[n]) \defeq \sum_{n\in \Z} z^nb_n[n].\]
Hence it has a \(\TT\)-spectrum. 
The Green-Julg theorem asserts
 that the \(\TT\)-equivariant \(\K\)-theory of \(A\) is 
isomorphic to the \(\K\)-theory \(\K_*(A\rtimes \TT)\) of the 
cross-product. By Takai-Takesaki duality, this agrees 
with \(\K_*(B)\).

\begin{proposition}
Let \(B\) be a C*-algebra and
\(\sigma \in \Aut (B)\). Endow the cross-product
\(A \defeq B\rtimes_\sigma \Z\) with the
dual action of \(\TT \cong \widehat{\Z}\).
The automorphism induces an invertible linear map 
\(\sigma_*\colon \K_*(B) \to \K_*(B)\) and hence a 
\(\Laur\) module structure on 
\(\K_*(B)\). This module is naturally isomorphic to 
\(\K^\TT_*(A)\). 

In particular, \emph{if \(\K_*(B)\) is finite dimensional over \(\C\)},  then 
\[ \TTspec (B\rtimes \Z) = \textup{Spec}(\sigma_*),\]
with \(\textup{Spec}(\sigma_*)\) the set of eigenvalues
of the invertible linear map \(\sigma_*\in \End_\C\bigl( \K_*(B)\bigr)\). 
\end{proposition}

\begin{proof}

This follows from Blackadar Proposition 11.8.3,
which asserts that the isomorphism
\[ \K_*(B) \cong \K_*(B\ltimes \Z\ltimes \TT) = 
 \K_*(A\rtimes \TT) \cong \K^*_\TT(A)\]
of Takai-Takesaki duality and the Green-Julg theorem,
intertwines the group homomorphism
\( \sigma_*\) and the group homomorphism of sclar 
multiplication by \(X\in \Laur \cong \Rep (\TT)\).

Furthermore, if \(\K^*_\TT (A)\) has finite
dimension, then \(\TTspec (A)\) has nonzero annihilator 
ideal and the support is the set of non-zero
eigenvalues of the linear map \(X\) because it is the 
zero set of the minimal polynomial of the linear map \(X\).
\end{proof}

\begin{remark}
\label{rem:bs_duality}
Baaj-Skandalis duality (see 
\cite{Baaj-Skandalis:Duality}) is a functor 
\(\KK^\TT \to \KK^\Z\) which on objects sends a
\(\TT\)-C*-algebra \(B\) to the \(\Z\)-C*-algebra 
\(B\defeq A\rtimes \TT\), with the dual action and 
sends a \(\TT\)-equivariant *-homomorphism 
\(A \to A'\) to the (obvious) induced \(\Z\)-equivariant map 
\(B\defeq A\rtimes \TT \to B' \defeq A'\rtimes \TT\). 
Baaj and Skandalis extend this to a 
natural \emph{isomorphism}
\[ \KK^\TT_*(A,A') \cong \KK^\Z_*(A\rtimes \TT, A'\rtimes \TT)\] 
of equivariant \(\KK\)-groups. Note that this 
transformation maps an induced space 
\(X = \TT\times_H Y\) for some \(H\)-space \(Y\) and a 
closed subgroup \(H\) of \(\TT\) to the \(\Z\)-C*-algebra 
\[ C_0(X)\rtimes \TT\cong C_0(Y)\rtimes H\]
with an appropriate dual action of \(\Z\). The important point is that this 
\(\Z\)-action \emph{factors 
through a periodic action}, \emph{i.e.} factors through the homomorphism 
\(\Z\to H \cong \Z/n\) for some \(n\), and a \(\Z/n\)-action. 

Under the Baaj-Skandalis transformation, the 
\(\TT\)-spectrum of a \(\TT\)-C*-algebra \(A\) corresponds, as we 
have observed above, to the spectrum, in the usual sense, 
of the endomorphism of \(\K_*(A)\) by the generator 
\(1\in \Z\) of the \(\Z\)-action. Hence if the \(\Z\)-action 
is periodic, then the corresponding 
linear map has finite order, and hence its spectrum consists of 
roots of unity in the circle. This is, roughly, then, the 
counterpart of the situation in the 
first section, in the category \(\KK^\Z\).  
\end{remark}

For instance let \(A = \C\) with the trivial automorphism. 
Applying the proposition gives that the
 \(\TT\)-spectrum of \(C^*(\Z)\) with its dual 
action of \(\TT\) is the single point \(\{1\} \subset \C^*\).

\begin{example}
The \(\TT\)-spectrum of the irrational rotation algebra 
\(A_\theta \defeq C(\TT)\rtimes_{R_\theta} \Z\)
 with the dual action of \(\TT\)
 is also \( \{1\} \) because \(\sigma_*\) is the identity map 
 on \(\K^*(\TT) \). 
\end{example}

\begin{example}
\label{ex:oa}
Let \(A\) be an integer \(n\)-by-\(n\) matrix with entries either 
\(0\) or \(1\) and assume for simplicity that \(A\) is 
invertible over \(\C\). Then the \(\TT\)-spectrum of the 
associated Cuntz-Krieger algebra \(O_A\) is 
the set of eigenvalues of \(A\). Indeed, 
\(O_A\cong F_A\ltimes \Z\) where \(F_A\) is an 
appropriate 
AF-algebra, and \(\cong\) means Morita equivalence. 
It is well-known and easily checked from the Bratteli diagram, 
that the \(\K\)-theory of \(F_A\) is \(\cong \C^n\), and 
the action of \(\Z\) on it is by the matrix \(A\). Hence 
the \(\TT\)-spectrum of \(O_A\) is the spectrum of \(A\). 
\end{example}

\begin{corollary}
\label{cor:cuntz}
The Cuntz-Krieger algebra \(O_A\) is not 
\(\KK^\TT\)-equivalent to any commutative 
\(\TT\)-C*-algebra as soon as the integer matrix 
\(A\) has some eigenvalue of modulus \(\not= 1\).  
\end{corollary}

This happens for instance if 
\(A = \left[ \begin{matrix} 1 & 1\\ 1 & 0\end{matrix}\right]\). 

For the benefit of the reader (the result is well-known) we 
prove the following. 

\begin{lemma}
Both \(O_A\) and \(O_A\rtimes \TT \cong F_A\) are in the 
boostrap category \(\mathcal{N}\). 
\end{lemma}

\begin{proof}
\(F_A\) is an AF algebra so is in \(\mathcal{N}\). 
The Baum-Connes conjecture for \(\Z\) is the 
statement that \(C_0(\R)\) with the \(\Z\)-action by 
translation is \(\KK^\Z_1\)-equivalent
to \(\C\). It follows from this that  
\(O_A = F_A\rtimes \Z\) is 
\(\KK\)-equivalent to \(C_0(\R, F_A)\rtimes \Z\). 
There is an exact sequence 
\[ 0 \rightarrow S\otimes F_A \otimes \Comp 
\rightarrow C_0(\R, F_A)\rtimes \Z \rightarrow
F_A\otimes \Comp \to 0\]
of C*-algebras, obtained by evaluating functions on \(\R\) at the 
integer points \(\Z\subset \R\), a closed and \(\Z\)-invariant subset, 
and using \(C_0(\Z)\rtimes \Z \cong \Comp\). Since 
\(\Comp\) is \(\KK\)-equivalent to \(\C\) both ends are 
in the boostrap category. Hence 
\(C_0(\R, F_A)\rtimes \Z\) 
is also. 
\end{proof}

\begin{remark}
We have actually proved something stronger than 
Corollary \ref{cor:cuntz}, for we have shown that 
the \(\TT\)-equivariant \(\K\)-theory of \(O_A\) is 
not isomorphic in the category of \(\Laur\)-modules 
to the \(\TT\)-equivariant \(\K\)-theory of any 
locally compact Hausdorff \(\TT\)-space. 

\end{remark}

We close this section with some further remarks 
on \(\TT\)-equivariant \(\K\)-theory of Cuntz-Krieger 
algebras, to see \(\TT\)-spectra in a dynamical 
perspective. 

Up to now we have considered \(\K^\TT_*(A)\defeq 
\K_0^\TT(A) \oplus \K^1_\TT(A)\) as simply 
a \(\Laur\)-module without taking into consideration 
the grading. If we consider \(\K^\TT_*(A)\) as a 
\(\Z/2\)-graded 
\(\Laur\)-module, then an invariant of it -- assuming it finite dimensional 
over \(\C\) -- is the rational function 
\begin{equation}
\label{eq:char}
 \textup{char}_A(t) \defeq \frac{\det(1-tX_+)}{\det(1-tX_-)}
 \end{equation}
where \(X_{\pm}\) denotes the action of the generator 
\(X\) on \(\K_{0/1}^\TT (A)\).

If \(A\) and \(B\) are 
\(\KK^\TT\)-equivalent, they have the same 
rational function \eqref{eq:char}. 

The following elementary result about (grading-preserving) 
linear transformations \(X\) 
on a \(\Z/2\)-graded vector space 
can be found in the appendices 
to Hartshorne's book \cite{Hartshorne}:

\begin{equation}
\label{eq:hartshorne}
\textup{char}_A(t) = \exp \bigl( \sum_{n=1}^\infty \textup{trace}_s(X^n) \frac{t^n}{n}\bigr)
\end{equation}
holds, where \(\trace_s\) is the graded trace, the 
difference of the traces of \(X\) acting on \(\K_1^\TT(A)\) and 
\(\K_0^\TT(A)\).

We now specialize to the following situation: 
let \(\phi_T\colon \TT^n \to \TT^n\) be a 
linear automorphism, where 
\(T\in \textup{GL}_n(\Z)\). We assume that 
\(T\) is self-adjoint, so it is diagonalizable over 
\(\C\) with real, nonzero eigenvalues. We can form the 
cross-product \(A \defeq C(\TT^n)\rtimes_{\phi_T} \Z\), which 
is a \(\TT\)-C*-algebra. By the Lefschetz fixed-point theorem, 
\[ \trace_s \bigl( (\phi_T^*)^n\bigr) = (-1)^k \, P_n(\phi_T),\]
because the sign of \(\det (1-T)\) is 
\( (-1)^k\) where \(k\) is the number (including multiplicities) 
of eigenvalues \(\lambda \) 
of \(T\) with \(\lambda >1\). Here \(\trace_s(\phi_T)\) is the graded 
trace of  
the action of \(\phi_T\) on \(\K^*(\TT^n)\), and \(P_n(\phi_T)\) is the 
number of periodic points of order \(n\). 

Putting things together, we see that 
\[\frac{\det (1-tX_+)}{\det(1-tX_-)} = 
\exp \bigl( (-1)^k\cdot \sum_{n=1}^\infty P_n(\sigma) \frac{t^n}{n}\bigr).\]
The right-hand-side is called the 
\emph{Artin-Mazur zeta function} of the map \(\phi_T\) (see 
\cite{Artin-Mazur:Zeta}.) 

To be explicit, if \(n = 2\) and 
 \(T =\left[\begin{matrix} 1 & 1\\1 & 0\end{matrix}\right]\), so \(k=1\), 
\(X_+ = \Id\) and \(X_-\) acts as \(T\) on \(\K^1(\TT^2) \cong 
\C^2\) and so 
\( \textup{char}_{C(\TT)\rtimes_{\phi_T}\Z } (t) = t^2-t-1\) and 
\[ \TTspec \bigl( C(\TT^2)\rtimes_{\phi_T}\Z \bigr) = \{ 1, \frac{1\pm \sqrt{5}}{2}\}.\]
Note that this yields another example of a \(\TT\)-C*-algebra 
not \(\KK^\TT\)-equivalent to a commutative one.

\section{\(\TT\)-equivariant \(\KK\)-theory of smooth manifolds and localization}

If \(R\) is any commutative (unital) ring 
 then any free, finitely generated \(R\)-module 
\(M\) has a well-defined \emph{rank}, and any 
\(R\)-module self-map of \(M\) has a well-defined 
\emph{trace}. We denote these invariants 
by \(\rank_R (M)\) and \(\trace_R (L)\) respectively, 
so that in particular \(\trace_R (\Id) = \rank_R(M)\). 

We will be mainly interested in the case where 
\(R = \Laur\) or a localization of \(R\). 

Consider \(\Laur\) as regular (rational) functions on \(\C^*\). 
In algebraic geometry, if one wants to study the behavior of 
a variety near a point \(z\in \C^*\), then one considers 
the set \(S\) of functions which are nonzero at \(z\), and 
\emph{localizes} \(\Laur\) with respect to this multiplicative set (a 
subset of a ring is a multiplicative set if it includes the unit \(1\) and 
is closed under multiplication.) 

This means that we invert 
all functions which are in \(S\), \emph{i.e.} invert functions 
which do not vanish at \(z\).  We therefore
get all rational functions which are regular at \(z\): 
\[ \Laur_z \cong \{f\in \C(X) \mid f= \frac{h}{g}, \; g(z)\not= 0\}.\]
This is a \emph{local ring}: it has a unique maximal ideal, 
the ideal of \(f\in \Laur_z\) such that \(f(z) = 0\), and 
any \(f\in \Laur\) such that \(f (z) \not= 0\) is invertible in 
\(\Laur_z\). 

Note also that \(\Laur\) embeds in its 
localization(s). 

Localization can be defined for
 any commutative ring \(R\) with no 
zero divisors, at a multiplicative 
subset \(S\) (like the complement of a 
prime ideal) by considering the elements
\(\frac{r}{s}\) in the ring of fractions of \(R\), 
such that \(s\in S\). In this situation, 
\(R\) embeds in its localization. 

For rings with zero divisors, localizations can 
still be defined, but the map from the original ring 
to its localization need not any longer be 
injective. Any element \(r\in R\) such that there 
exists \(s\in S\) so that \(rs= 0\), is  
is killed by localization at \(S\). 

The prime ideals of the localization of a ring \(R\) at 
\(S\) correspond to the prime ideals of \(R\) which 
do not intersect \(S\).

The `localizations' \(\Laur_z\) just discussed, are the 
\emph{stalks} of a sheaf of rings over \(\C^*\) 
with the Zariski topology. For most of this paper, we 
will not use the stalks, but the values of the sheaf on 
Zariski open sets. To fix notation and terminology, we 
state the definition formally. 

\begin{definition}
Let \(f\in \Laur\). The \emph{localization} of \(\Laur\) at 
the Zariski open 
\(U_f\defeq \C^*- Z_f\) is the ring obtained 
from \(\Laur\) by inverting all powers of \(f\).  
We denote by \(\Laur_f\) the localization of 
 \(\Laur\) at \(U_f\). 
 The assignment  \(U_f \mapsto \Laur_f\) 
 defines a sheaf on \(\C^*\) with the Zariski topology. 
 The stalks of this sheaf are denoted 
 \(\Laur_z\) and are as discussed above. 
 \end{definition}

Note that inverting \(f\) automatically inverts all divisors of \(f\) and 
hence inverts all polynomials which do not 
vanish on \(U_f\), since the roots of such a polynomial are all 
roots of \(f\), which implies it is a divisor of some positive power of \(f\). 

Hence \(\Laur_f\) is simply the ring of 
\emph{regular rational functions on \(U_f\)}.

Modules over a ring \(R\) can also be localized at multiplicative 
subsets of \(R\), by setting 
\[ M_S \defeq M\otimes_R R_S\]
where \(R_S\) is the localization of \(R\) at \(S\). In the 
case of interest, where \(R = \Laur\), we denote by 
\(M_f\) the localization of a \(\Laur\)-module at 
\(S\defeq \{1, f, f^2, \ldots\}\) (that is, at \(U_f\).) The 
important point is that \emph{localization of a module
at \(U_f\) kills torsion supported 
in \(Z_f\)}. If \(M\) is a torsion module with 
finite support, then \(M_f = 0\) if \(f\) vanishes 
on the support. More generally, of course, 
if the support of the torsion submodule of a finitely generated 
module \(M\) is
\(Z_f\) then localizing \(M\) at \(U_f\) kills the torsion part, and the 
localization of the free part is free (over \(\Laur_f\).) (Recall that since 
\(\Laur\) is a principal ideal domain, every finitely generated 
\(\Laur\)-module splits uniquely into a torsion and a free module.) 

We now consider the case where the module \(M\) has the form 
\( M = \K^*_\TT (X)\) where \(X\) is a \(\TT\)-space. More generally, 
we may consider any \(\KK^\TT\)-group, \emph{i.e.} any 
\(\KK_*^\TT(A,B)\), for \(A\) and \(B\) \(\TT\)-C*-algebras, with 
its \(\Laur\)-module structure. If \(f\in \Laur\) we may localize 
any such module at \(f\), yielding \(\KK^\TT_*(A,B)_f\). Localization 
is obviously compatible with the \(\Z/2\)-gradings, the intersection 
product (composition in \(\KK^\TT\)) and the external product. 
In particular we may speak of \(\KK^\TT_f\)-equivalence and so on. 

\begin{remark}
Localization in \(\K\)-theory is slightly different from 
localization in equivariant cohomology as in \cite{Atiyah-Bott:Moment}. 
\begin{enumerate}
\item The coefficient ring \(\Laur = \KK^\TT(\C, \C)= \KK^\TT_*(\C, \C)\) 
we use is trivially graded, while the cohomological analogue  
\(\textup{H}^*_\TT (\pnt) \defeq \textup{H}^*(B\TT) \cong \C[u]\) 
\(\C[u]\) is \(\Z\)-graded with \(\textup{deg}(u) = 2\). 
Atiyah's Completion 
Theorem relates the two rings: equivariant cohomology is the 
\(I\)-adic completion of \(\Laur\) with respect to the ideal 
\(I \defeq \langle X-1\rangle\) corresponding to \(1 \in \C^*\). Supports of 
\(\C[u]\)-modules, like for example \(H^*_\TT(X) \defeq H^*(E\TT\times_\TT X)\) for a \(\TT\)-space \(X\), are contained in \(\C\) instead of \(\C^*\). If the 
modules are graded, then 
their supports are always either all of \(\C\) or are 
\(\{0\}\), because they must be a cone (see \cite{Atiyah-Bott:Moment}).
 Therefore the cohomological 
analogue of \(\TTspec\) is rather trivial: the support of the 
torsion submodule of \(H^*_\TT(X)\) must be \(\{0\}\) and 
after localizing at \(\C^*\defeq  \C- \{0\} \) we get a free module.
\item  After localizing \(\textup{H}^*_\TT (X)\) by localizing, separately, 
its even and odd parts, the integer gradation on the module becomes 
lost; the \(\Z/2\)-grading is not lost, however. 
\end{enumerate}
Both of these facts would 
seem to support the idea that \(\K\)-theory responds somewhat 
better to localization. 
\end{remark}

%The Localization Theorem of
%Atiyah and Segal (see Theorem 1.1 of \cite{Atiyah-Segal:Index}).  

We are now going to refine some of our results from 
the first section about equivariant \(\K\)-theory of 
spaces, using localization. See also \cite{Atiyah-Segal:Index}, 
for some overlapping results.

We begin by discussing the issue of finite generation, 
which, importantly, implies that 
the torsion submodule of \(\K^*_\TT (X)\) has 
finite spectrum.
Graeme Segal has proved the following.

\begin{lemma} (\cite{Segal}, Proposition 5.4). 
\label{lem:Segals_lemma}
If \(X\) is a smooth compact \(\TT\)-manifold, then 
\(\K^*_\TT (X)\) is a finitely generated \(\Laur\)-module. 
\end{lemma}
The following is a useful geometric counterpart of Segal's 
lemma.

\begin{lemma}
\label{lem:products}
 For a compact manifold \(X\) with smooth \(\TT\)-action, 
 there are only finitely many points \(t\in \TT\) which 
 fix some point of \(X- F\), where \(F\subset X\) is the 
 stationary set. 
 
Hence if \(f\in \Laur\) is a polynomial which vanishes 
 on these points, then 
 \(\K^*_\TT \bigl( W\times (X- F)\bigr)_f = 0\) for 
 any locally compact \(\TT\)-space \(W\). 
\end{lemma}

\begin{proof}
For the first statement, since \(F\) is a smooth 
submanifold of \(X\) it has a normal bundle \(\nu\), 
which is a \(\TT\)-equivariant real vector bundle. 
This may be identified with the orthgonal 
complement of \(\Tvert F\) in \(\Tvert X|_F\) with 
respect to any \(\TT\)-invariant Riemannian 
metric. Since the fixed-point set of \(t\in \TT\) 
in \(\Tvert_x X\) (for \(x\in F\) ) is exactly 
\(\Tvert F\), \(t\) fixes no nonzero vector in 
\(\nu\). 

Let \(U\) be the 
 corresponding \(\TT\)-invariant open neighbourhood of 
 \(F\). Since \(\TT\) acts freely on \(\nu -  0\) it 
 acts freely on \(U- F\). We can cover 
 the compact \(X- U\) by finitely many 
 open slices \(W_i\subset X- U\), centred, say 
 at points \(x_i\), and if \(x\in X- U\) is 
 any point, then \(\TT_x\subset \TT_{x_i}\) follows for 
 \(x\in W_i\). Since \(\TT_x = \{1\}\) for 
 \(x\in U\), \(\bigcup_{x\in X- F} \TT_x\)
 is a finite set as claimed. 
 
 If \( (w, x) \in W\times X- F\) then of 
 course \(\TT_{(w,x)}\subset \TT_x\). It follows that 
 if \(f\in \Laur\) vanishes on \(\bigcup_{x\in X-F}\TT_x\)
  then it annihilates 
 the image of \(\K^*_\TT (Y) \to \K^*_\TT \bigl(W\times (X- F)\bigr)\) 
 for any pre-compact \(Y\subset W\times (X- F)\),
  \emph{c.f.} the arguments  in 
 the first paragraph of the proof of Lemma \ref{lem:lemmas}. 
 Hence it annihilates \(\K^*_\TT \bigl(W\times (X- F)\bigr)\). 
\end{proof}

\begin{example}
\label{ex:many_points}
Consider \(\TT\times \N\) with the \(\TT\)-action of 
Example \ref{ex:points}. Let \(X\) be the one-point 
compactification of \(X\times \N\), with \(\TT\)-action 
the canonical extension of the action on \(X\times \N\) 
(fixing the point at infinity). Then \(X\) is a compact 
space but there are infinitely 
many distinct points \(t\in \TT\) which fix some 
point of \(X-F\). The equivariant \(\K\)-theory is 
zero in dimension \(1\) and in dimension \(0\) is 
the \(\Laur\)-module 
\[ \K^0_\TT (X) \cong \Laur \oplus \bigoplus_{n\in \N} \, \Laur/ (f_n)\]
where \(f_n (X) = \prod_{\omega \in \Omega_n} X-\omega\).
The torsion submodule of \(\K^0_\TT (X)\) is not finitely generated 
and has support \(\C^*\). Thus both Lemma \ref{lem:Segals_lemma}
and Lemma \ref{lem:products} fail for \(X\) due to a lack of a 
`collaring' for the stationary set. 
\end{example}

\begin{lemma}
\label{lem:fg_exact}
Let \(A,B\) and \(C\) be \(\Laur\)-modules, 
\(\alpha \colon A \to B\) and \(\beta\colon B \to C\) 
module maps, such that the sequence
\[ 0 \longrightarrow \textup{im}(\alpha) 
  \longrightarrow B \longrightarrow \ker(\beta) \longrightarrow 0\]
is exact and \(A\) and \(C\) are finitely generated. Then 
\(B\) is finitely generated. 

\end{lemma}

\begin{proof}
This reduces immediately to whether \(\ker(\beta)\) and 
\(\textup{im}(\alpha)\) are finitely generated; the latter is 
obvious and the former follows from the fact that any 
submodule of a finitely generated \(\Laur\)-module is finitely 
generated, because \(\Laur\) is Noetherian. 
\end{proof}

\begin{corollary}
\(\K^*_\TT(X-F)\) is finitely generated for any smooth and 
compact 
\(\TT\)-manifold \(X\). 
\end{corollary}

\begin{proof}
By Lemma \ref{lem:Segals_lemma}
 \(\K^*_\TT (X)\) and \(\K^*_\TT (F) \) are
finitely generated \(\Laur\)-modules (see Remark \ref{rem:stationary_smooth}.) 
The result then follows from Lemma \ref{lem:fg_exact}, for \(\K^*_\TT (X-F)\) 
fits into a \(6\)-term exact sequence with the other 
terms \(\K^*_\TT (F)\) or \(\K^*_\TT (X)\). 
finitely generated. 
\end{proof}

\begin{remark}
\label{rem:stationary_smooth}
The stationary set of a smooth \(\TT\)-action is smooth: 
a choice of a \(\TT\)-invariant Riemannian 
metric yields, at every \(x\in F\), an 
exponential map, \(\Tvert_x X \to X\), which is \(\TT\)-equivariant 
and is a diffeomorphism in a small metric ball around the 
origin in \(\Tvert_x X\). Therefore \(\exp_x\) intertwines 
(an open subset of) the stationary set of the \emph{linear} action 
of \(\TT\) on \(\Tvert_x X\), to (an open subset) of the 
stationary set \(F\). This yields a 
\(\TT\)-equivariant smooth manifold chart around \(x\) in \(F\).
\end{remark}

It follows that \(\K^*_\TT (X-F)\) has a finite \(\TT\)-spectrum, equivalently, 
a nonzero annihilator ideal, because it is finitely generated and torsion.

 We will discuss \(\TT\)-equivariant Poincar\'e duality 
 for smooth manifolds 
in greater depth later; for now, the following statement 
is useful for proving certain things quickly.

\begin{theorem} 
\label{thm:duality}
Let \(X\) be a smooth and compact \(\TT\)-manifold and 
\(D\in \KK^\TT(C(\Tvert X), \C)\) the class of the 
Dirac operator on the almost-complex \(\TT\)-manifold 
\(\Tvert X\). Then cup-cap product with 
\(D\) determines a natural family of isomorphisms  
\[ \KK^\TT_*( C(X)\otimes A, B) \cong \KK^\TT_*(A, C_0(\Tvert X)\otimes B)\]
for all \(\TT\)-C*-algebras \(A,B\). 
\end{theorem}

Theorem \ref{thm:duality}
is due to \cite{Connes-Skandalis:Longitudinal} in the 
non-equivariant setting. See also Kasparov \cite{Kas} in the equivariant setting and his references. For a modern 
treatment of equivariant Poincar\'e duality see 
\cite{Emerson-Meyer:Dualities}.

It follows from Poincar\'e duality
that if \(W\) and \(Z\) are compact smooth
 \(\TT\)-manifolds, then 
\(\KK^\TT_*(C(W), C(Z))\) is a finitely generated 
\(\Laur\)-module. Indeed, duality reduces us to 
proving that \(\K^*_\TT (\Tvert W\times Z)\) is 
finitely generated, which follows from Lemma \ref{lem:vector_bundles} below. 
From this,
 and consideration of the \(6\)-term exact sequence 
 associated to \(F\subset X\), we deduce that
 the modules 
 \emph{e.g.} \(\KK^\TT_*(C_0(X-F), C(F))\) of morphisms in \(\KK^\TT\) between 
 any two of \(C(X), C(F)\) and \(C_0(X-F)\), are finitely generated. 
 
%Note that \(\KK^\TT_*(C(W), \C)\) is not finitely generated 
% if \(W\) is the Cantor set, with trivial \(\TT\)-action. 

\begin{lemma} 
\label{lem:vector_bundles}
If \(X\) is a compact smooth \(\TT\)-manifold and 
\(V\to X\) is a real \(\TT\)-equivariant vector bundle on \(X\), then 
\(\K^*_\TT (V)\) is finitely generated. Moreover, if \(\TT\)
acts freely on \(V-0\) then 
the restriction map 
\[ \K^*_\TT (V) \to \K^*_\TT (X)\]
induces an isomorphism after localizing at 
\(\C^*-\{1\}\).

\end{lemma}

In particular, \(\TTspec (V) = \TTspec (X)\cup\{1\}\) if \(\TT\) acts 
freely on \(V-0\). A good example is \(V = \mathrm{T}X\) for a 
compact smooth \(\TT\)-manifold \(X\) where stationary points of the 
action are isolated, that is, where there is only a finite number of 
them. Fixing a \(\TT\)-invariant Riemannian metric, any nonzero 
tangent vector which is fixed by the \(\TT\)-action results in a  
geodesic which is point wise fixed by the action, contradicting that 
the stationary set consists of finitely many points. Thus the 
\(\TT\)-action on nonzero tangent vectors is free.

\begin{proof}
Fix a \(\TT\)-invariant metric on \(V\) and consider 
the exact sequence 
\[ 0 \longrightarrow C_0(D_V)\longrightarrow 
C_0(\overline{D}_V) \longrightarrow C(S_V)\longrightarrow 0\]
where \(D_V\) is the open disk bundle, 
\(\overline{D}_V\) the closed disk bundle, and 
\(S_V\) the sphere bundle. Since 
\(\overline{D}_V\) is \(\TT\)-equivariantly proper 
homotopy equivalent to \(X\), which is a compact 
smooth manifold, and since \(S_V\) is also a 
compact smooth manifold, it follows from considering 
the associated \(6\)-term exact sequence and Lemma 
\ref{lem:fg_exact} that 
\(\K^*_\TT (V) \cong \K^*_\TT (D_V)\) is finitely 
generated. 

If \(\TT\) acts freely on \(V-0\) then it acts freely on 
\(S_V\) and hence \(\TTspec (S_V) \subset \{1\}\). 
Therefore, localizing at \(\C^*-\{1\}\) kills 
\(\K^*_\TT(S_V)\) and the claim follows.

\end{proof}

\begin{remark}
Suppose that \(V\) carries a \(\TT\)-equivariant 
\(\K\)-orientation. The \emph{Euler class} 
\(e_V\in \K^{-\dim (V)}(X)\) of 
\(V\) can be defined as the restriction to \(X\) (the 
zero section in \(V\)) of the Thom class for \(V\), in 
\(\K^{-\dim (V)}_\TT (V)\). The Thom class generates 
\(\K^*_\TT (V)\) as a 
free rank-one \(\K^*_\TT (X)\)-module. It follows 
that the restriction map 
\( \K^*_\TT (V) \to \K^*_\TT (X)\) identifies, under 
\(\K^*_\TT (V)\cong \K^*_\TT (X)\), with the 
map 
\[ \K^*_\TT (X) \to \K^*_\TT (X), \; \;\;\; \xi \mapsto \xi \cdot e_V.\]
It follows then from Lemma \ref{lem:vector_bundles} 
that \(e_V\) becomes an invertible after we localize
at \(\C^*-\{1\}\), that is, \(e_V\) is an invertible in the 
ring \(\K^*_\TT (X)_f\) where \(f(X) = X-1\).

  This fact is used frequently in 
connection with characteristic class computations in 
the work of Atiyah and Segal and in 
Atiyah and Bott's paper \cite{Atiyah-Bott:Moment}. 

\end{remark}

\begin{theorem}
\label{thm:restriction_to_stationary}
Let \(X\) be a compact smooth 
\(\TT\)-manifold and \(F\subset X\) the stationary set.
Let 
\[\Omega\defeq \{t\in \TT\mid tx = x\; \textup{for some}\; x\in X- F\}.\]
\(\Omega\) is finite. 
Let \(f\in \Laur\) be a polynomial 
vanishing on \(\Omega\). 
 Then \(C_0(X- F)\) is \(\KK^\TT_f\)-equivalent to 
 the zero \(\TT\)-C*-algebra, and 
the localization \(\rho_f \in \KK^\TT(C(X), C(F))_f\) of the 
restriction morphism \(\rho \in \KK^\TT(C(X), C(F))\), is invertible
(in \(\KK^\TT_f\)). 

\end{theorem}

This theorem is similar to Proposition 1.5 of 
\cite{Atiyah-Segal:Index}.

\begin{proof}
To prove that \(C_0(X- F)\) is \(\KK^\TT_f\)-equivalent 
to zero 
it suffices to prove that \(\KK^\TT_*(C_0(X- F), 
C_0(X- F))_f\) is the zero module over \(\Laur_f\). 
By Lemma \ref{lem:products}, if \(f\) vanishes on \(\Omega\) then 
\(\K^*_\TT \bigl(\Tvert F\times (X- F)\bigr)_f = 0 = 
\K^*_\TT \bigl(\Tvert X\times (X- F)\bigr)_f\) and by 
Poincar\'e duality for respectively \(F\) and \(X\) this implies that 
\begin{equation}
\label{eq:vanish}
\KK^\TT_*( C(F), C_0(X- F))_f = 0,\;  \;\;\;\;
\KK^\TT_*( C(X), C_0(X- F))_f=0.\end{equation}
Using the \(6\)-term exact sequence applied to 
the first variable, we deduce that 
\[  \KK^\TT_*( C_0(X- F), C_0(X- F))_f= 0\]
too. Thus, \(C_0(X- F)\) is \(\KK^\TT_f\)-equivalent 
to the zero \(\TT\)-C*-algebra as claimed. 

(We could not use Poincar\'e duality directly for 
\(C_0(X- F)\) because it is non-compact, 
and duality works differently for non-compact
spaces.) 

Now from the \(6\)-term exact sequence, and the 
fact just proved that \(C_0(X- F)\) is 
\(\KK^\TT_f\)-equivalent to zero, 
the map 
\begin{equation}
\label{eq:restriction_isomorphism}
 \KK^\TT_*( A, C(X))_f \xrightarrow{\cdot \otimes_{C(X)} \rho} 
 \KK^\TT_*( A, C(F))_f
 \end{equation}
induced by restriction to \(F\) is an isomorphism for any 
\(\TT\)-C*-algebra \(A\). Now use the 
Yoneda lemma: 
set \(A \defeq C(F)\) and find a pre-image 
\(\alpha\in \KK^\TT(C(F), C(X))_f\) 
of the identity morphism in \( \KK^\TT_*(C(F), C(F))_f\).
Then the composition in \(\KK^\TT_f\) 
\[ C(F) \xrightarrow{\alpha} C(X)\xrightarrow{\rho} C(F)\]
is the identity by the definitions, and the composition 
\[ C(X) \xrightarrow{\rho} C(F) \xrightarrow{\alpha} C(X)\]
is therefore multiplication by an idempotent \(\gamma \defeq \rho\otimes_{C(F)} \alpha \in 
\KK^\TT( C(X), C(X))_f\). To show that \(1-\gamma = 0\) set 
\(A \defeq C(X)\), and observe that this is mapped to zero under 
composition with \(\rho\), \emph{i.e.} under the map \eqref{eq:restriction_isomorphism}. 
Since the latter is an isomorphism after localization, 
 \(1-\gamma = 0\). 

\end{proof}

\begin{remark}
While a properly formulated version of Theorem 
\ref{thm:restriction_to_stationary} \emph{should} be 
true without smoothness assumptions (\emph{c.f.} 
Theorem \ref{thm:commutative_case}, which does not 
use such an assumption), we have not 
pursued it since we are mainly interested in smooth 
manifolds anyway, and because Example 
\ref{ex:many_points} shows that away from smooth 
manifolds, \(\TTspec (X-F)\) may not be finite, which 
makes it more difficult to formulate a theorem. 

\end{remark}

%\begin{proposition}
%\label{prop:homological_restriction_to_stationary} 
%In the above notation, 
%\[ \supp \bigl( \KK^\TT_*(C_0(X-F), C_0(X-F))\bigr) 
%\subset \TTspec (X-F) .\]
%In particular, if \(f\) vanishes on \(\TTspec (X-F)\cup \{1\}\) then 
%\(\KK^\TT_*(C_0(X-F), C_0(X-F))_f = 0\). 
%\end{proposition}

%Otherwise put, if \(f\in \Laur\) annihilates \(\K^*_\TT (X-F)\), and 
%vanishes at \(1\in \C^*\), then 
%it also annihilates \(\KK^\TT_*(C_0(X-F), C_0(X-F))\). 

%\begin{lemma}
%\label{lem:spectra_of_vector_bundles}
%Let \(V\) be a real \(\TT\)-equivariant vector bundle over 
%a locally compact \(\TT\)-space \(B\); assume \(V\) has no nonzero 
%\(\TT\)-fixed-vectors. Then restriction 
% \(C_0(V) \to C_0(B)\) to the zero section of \(V\) induces an 
% isomorphism after localizing at \(\C^*- \{1\}\). Thus,
% \[\K^*_\TT (V)_f\longrightarrow \K^*_\TT (B)_f\]
% is an isomorphism for any \(f\in \Laur\) such that \(f(1) = 0\). 
 
% In particular, \(V\) and \(B\) have the same \(\TT\)-spectrum, 
% with the possible exception of the point \(1\). 
%\end{lemma}

%\begin{proof}
%Consider the \(6\)-term exact sequence for the pair \((\overline{D}_V, S_V)\) 
%as in the proof of Lemma \ref{lem:finite_generation}. Since 
%\(\TT\) acts freely on \(S_V\), \(\TTspec (S_V) \subset \{1\}\subset \C^*\). 
%Hence localizing at \(f(X) \defeq X-1\) kills \(\K^*_\TT (S_V)\) and 
%the result follows. 
%\end{proof}

\begin{corollary}
\label{cor:UCT}
Let \(D\) be a \(\TT\)-C*-algebra in the boostrap category, such that 
\(D\rtimes \TT\) is also in the boostrap category, let \(X\) be a 
smooth, compact \(\TT\)-manifold, and \(f, \Omega\) be as in 
Theorem \ref{thm:restriction_to_stationary}.
 Then 
\begin{enumerate}
\item\( \KK_*^\TT ( C(X), D)_f \cong \Hom_{\Laur_f}\bigl( \K^*_\TT (X)_f, \K_*^\TT (D)_f\bigr)\)
\item \( \KK^\TT_*(\C, C(X)\otimes D)_f\cong \K^*_\TT(X)\otimes_{\Laur_f}\K_*^\TT(D).\) 
\end{enumerate}

\end{corollary}

\begin{proof}
The class of \(\TT\)-spaces \(X\) for which both theorems hold 
(in \(\KK^\TT_f\))  is closed under \(\KK^\TT_f\)-equivalence so we may 
replace \(X\) by \(F\) by Theorem \ref{thm:restriction_to_stationary};
since \(F\) is a trivial \(\TT\)-space, \(\KK^\TT_*(C(F), D)\cong 
\KK_*(C(F), D\rtimes \TT)\) by the Green-Julg theorem, and by the 
UCT this is isomorphic to \(\Hom_\C\bigl( \K^*(F), \K_*(D\rtimes \TT)\bigr)
\cong \Hom_\C\bigl( \K^*(F), \K_*^\TT(D)\bigr)\). This implies the 
corresponding isomorphisms after localization. Now \(\K^*_\TT(X)_f\cong 
\K_\TT^*(F)_f \cong \bigl( \K^*(F)\otimes \Laur\bigr)_f \cong 
\K^*(F)\otimes \Laur_f\) and hence 
\(\Hom_{\Laur_f}\bigl( \K^*_\TT (X)_f, \K_*^\TT (D)_f\bigr)
\cong \Hom_{\Laur_f}\bigl( \K^*(F)\otimes \Laur_f, \K^*_\TT(D)_f\bigr)
\cong \Hom_\C\bigl( \K^*(F), \K^*_\TT(D)_f\bigr)\) which proves the 
first statement. The second follows similarly (see the proof of Lemma \ref{lem:lef_for_trivial_actions}.) 

\end{proof}

We end this section with a fairly precise description of 
\(\K^*_\TT (X)\) for smooth \(\TT\)-manifolds, starting with the 
following result, which uses ideas of Baum and 
Connes (see \cite{Baum-Connes:finite}).

\begin{theorem}
\label{thm:bigtheorem}
Let \(X\) be a smooth, compact \(\TT\)-manifold, 
\(F\subset X\) the stationary set. 
 
For \(\gamma \in \TT\) we 
endow the \(\C\)-vector space \(\K^*\bigl(\TT\backslash (X^\gamma - F)\bigr)\) with the 
\(\Laur\)-module structure by evaluation \(\Laur \to \C\) at \(\gamma\). 
Then 
\begin{equation}
\K^*_\TT (X - F) \cong 
\oplus_{\gamma \in \TTspec (X- F)} \K^*\bigl(\TT\backslash (X^\gamma-F)\bigr)
\end{equation} 
as \(\Laur\)-modules. 
%\item There is a \(6\)-term exact sequence 
%\[\xymatrix{\oplus_{\gamma \in \TTspec (X- F)} \K^0(\TT\backslash X^\gamma-F) \ar[r]& \K^0_\TT(X)\ar[r] & \K^0(F)\otimes \Laur \ar[d]^{\partial_0}\\ \K^1(F)\otimes \Laur \ar[u]^{\partial_1} & \K^1_\TT(X)\ar[l] & \oplus_{\gamma \in \TTspec (X- F)} \K^1(\TT\backslash X^\gamma -F)\ar[l]}\]
%\end{itemize}
 \end{theorem}

\begin{remark}
The usual geometric effect of localization of \(\K^*_\TT (X)\)
 at \(\gamma \in \TT\) -- 
it annihilates the contribution of \(X-X^\gamma\), as we have 
seen -- is obviously nil 
in the case where \(\gamma = 1\). Thus Theorem 
\ref{thm:bigtheorem} goes further in this case, informing 
us that the stalk at \(1\) of the sheaf determined by 
\(\K^*_\TT (X-F)\) 
is \(\K^*\bigl(\TT\backslash (X-F)\bigr)\) (with \(\Laur\)-module structure 
by evaluation at \(1\in \C^*\).) 

\end{remark}

\begin{proof} 
Set \(\Omega \defeq \{ \gamma \in \TT \mid \gamma x = x \; \textup{some} \; x\notin F\} \subset \TTspec (X-F)\). \(\Omega \) is finite. 
We consider a theory defined on \(\TT\)-spaces (like 
\(X-F\)) which 
can be covered by a finite number of open \(H\)-slices,  
where 
\(H \subset \Omega\) is some subset. This class of spaces is clearly closed 
under passing to subspaces. If \(Z\) is such a space, let 
\[ \textup{F}(Z) \defeq 
\bigoplus_{\gamma \in \Omega} \K^*(\TT\backslash Z^\gamma)\]
with module structure evaluation of characters at 
\(\gamma\) in the corresponding summand. Observe that 
we may interpret this vector space as \(\K^*(\widehat{Z})\) where 
\[ \widehat{Z} \defeq \TT \backslash \{ (z,\gamma)\in Z\times \TT \mid \gamma z = z\}.\]
Indeed, the space \(\widehat{Z}\) fibres over \(\Omega\) with fibre 
\(\TT\backslash X^\gamma\) over \(\gamma\). 

If \(Y\subset Z\) is a closed \(\TT\)-invariant subspace of 
\(Z\) in our class, then \(\widehat{Y}\subset \widehat{Z}\) as a 
closed subspace, and \(\widehat{Z}-\widehat{Y} = 
\widehat{Z-Y}\). Hence an inclusion of a closed \(\TT\)-invariant 
subspace generates a corresponding 
\(6\)-term exact sequence and the theory \(\textup{F}\) is excisive.
 To show that it agrees with \(\K^*_\TT (\, \cdot \, )\) it is sufficient 
 then to verify this for an induced space \(U \cong \TT\times_H Y\). 
In this case \(\K^*_\TT (U)\cong \K^*_H(Y)\) as \(\Laur\)-modules, 
 where the \(\Laur\)-module action on \(\K^*_H(Y)\) factors through 
 the restriction \(\Laur \to \Rep (H)\) and the \(\Rep (H)\)-module 
 structure on \(\K^*_H(Y)\). By a result of Baum and Connes for 
 equivariant \(\K\)-theory of finite group actions (see 
 \cite{Baum-Connes:finite})  
 \[ \K^*_H(Y)\cong \bigoplus_{h \in H}\K^*(H\backslash Y^h),\]
 where the \(\Rep (H) \cong \Laur/(f_H)\)-module 
 structure on the right-hand-side is 
 by evaluation of characters at the points of \(H\) (here 
 \(f_H = \prod_{h\in H} X-h\) and \( (f_H)\) is the ideal 
 of \(\Laur\) generated by \(f_H\).) We are using the 
 fact that \(H\) is abelian, so that the centralizer of \(h\) in 
 \(H\) is \(H\). Localizing at \(\gamma \in \Omega\) yields
 zero unless \(\gamma \in H\), and in this case, 
  \[ \K^*_H(Y)_\gamma \defeq 
 \K^*_H(Y)\otimes_{\Laur} \Laur_\gamma 
 \cong \bigoplus_{h\in H} \left[ \K^*(H\backslash Y^h)\otimes_{\Laur} \Laur_\gamma \right].\]
 Now for each term on the right-hand-side, the tensor product is over the 
 evaluation map \(\Laur \to \C\) at \(h\). It follows that
  all terms in the sum on the right-hand-side vanish except for \(h = \gamma\).  
  The \(\Laur_\gamma\)-module structure on this term is evaluation of 
 polynomials at \(\gamma\). Thus, 
 \[ \K^*_\TT(U)_\gamma \cong \K^*(H\backslash Y^\gamma)_\gamma.\]
 Given that \(H\backslash Y^\gamma \cong \TT\backslash U^\gamma\), 
 the result follows.

\end{proof}

In particular, we now have an exact description of \(\TTspec (X)\) when 
\(X\) is a compact smooth manifold. 

\begin{corollary}
Let \(X\) be a compact smooth
 \(\TT\)-manifold with no stationary points. Then 
\[ \TTspec (X) = \{ \gamma\in \TT \mid \K^*(\TT\setminus X^\gamma)\not= 0\}.\]
\end{corollary}

Before the proof, we use Theorem \ref{thm:bigtheorem} 
determine the exact relation between 
the torsion submodule of \(\K^*_\TT (X)\) and the torsion module 
\(\K^*_\TT (X-F)\).

Let \(\textup{Tors}\bigl(\K^i_\TT(X)\bigr)\) be the torsion part of 
\(\K^i_\TT (X)\) and \(\textup{Free}\bigl( \K^i_\TT(X)\bigr)\) the 
free part. 
The \(6\)-term exact sequence associated 
to the stationary set \(F\subset X\) yields  
\emph{surjections}
\[\K^i_\TT (X- F) \to \textup{Tors}\bigl(\K^i_\TT(X)\bigr)\]
since the map \(\K^i_\TT (X)\to \K^i_\TT (F)\) vanishes on the 
torsion part, since \(\K^*_\TT (F)\) is free, 
and 
\emph{injections} 
\[ \textup{Free}\bigl( \K^i_\TT(X)\bigr) \to \K^i_\TT (F)\cong \K^i(F)\otimes \Laur,\]
since the map \(\K^{i+1}_\TT (X-F) \to \K^{i+1}_\TT(X)\) has range in the 
torsion subgroup.

We have the boundary maps 
\begin{equation}
\label{eq:boundary_maps}
\partial_i\colon \K^{i-1}(F)\otimes \Laur \longrightarrow
\K^i_\TT(X-F)\bigr)\end{equation}
and thus 
\[\textup{coker}(\partial_i) \cong \textup{Tors}\bigl(\K^i_\TT(X)\bigr), \;\;\;
\ker(\partial_{i+1}) \cong \textup{Free}\bigl(\K^i_\TT(X)\bigr).\]

Theorem \ref{thm:bigtheorem} and some geometric 
arguments (using smoothness) tells us more. 

\begin{corollary}
\label{cor:spectra_away_from_one}
If \(X\) is a smooth compact \(\TT\)-manifold, then 
the range of \(\partial_i \colon \K^i_\TT (F) \to \K^{i+1}_\TT (X-F)\) 
is supported at \(1\in \C^*\). 
Hence \(\partial_i\) factors through a map  
\[ \partial_i'\colon \K^i_\TT (F) \to \K^{i+1}(\TT\setminus X-F).\]
Thus 
\(\textup{Tors}(\K^i_\TT (X))_z\cong \K^i_\TT (X-F)_z\) for 
all \(z\in \TT-\{1\}\), and for the component at \(1\in \C^*\) we have 
\[\textup{Tors}(\K^i_\TT (X)_1) \cong 
\K^i (\TT\backslash X-F) \, / \, \textup{im}(\partial_{i+1}').\] 

%In other words, the the torsion modules 
%\(\textup{Tors}\bigl(\K^*_\TT (X)\bigr)\) and \(\K^*_\TT (X-F)\) 
%have the same stalks except possibly at 
%\(1\in \C^*\). 

Furthermore, the 
free modules \(\textup{Free}\bigl( \K^*_\TT(X)\bigr)\) and 
\(\K^*_\TT(F)\) have the same rank in each dimension.  

\end{corollary}

\begin{remark}
The Lefschetz fixed-point theorem discussed below 
implies that the difference in ranks of the free part of \(\K^0_\TT(X)\) and 
the free part of \(\K^1_\TT (X)\) equals the difference of ranks of 
the \(\C\)-vector spaces \(\K^0(F)\) and \(\K^1(F)\). The above 
statement is stronger, since it holds before taking differences. 
\end{remark}

The boundary maps in the \(6\)-term exact sequence of 
Theorem \ref{thm:bigtheorem} can be computed fairly 
precisely if \(X\) is a smooth manifold with smooth 
\(\TT\)-action, and this also proves the Corollary 
\ref{cor:spectra_away_from_one}. 

For the definition of correspondence, used below, see the 
discussion in \S \ref{sec:lef}. 

\begin{proof} (Of Corollary \ref{cor:spectra_away_from_one}). 
\(F\) is a closed, 
smooth submanifold of \(X\). 
Let \(\nu\) be the normal bundle of the 
stationary set \(F\subset X\); it can be endowed 
with a \(\TT\)-action and invariant Riemannian 
metric. Let \(\hat{\varphi}\colon \nu \to X\) the 
tubular neighbourhood embedding. 

Let \(S\nu\) be the sphere bundle of 
\(\nu\) and \(\pi\colon \nu \to F\) the bundle 
projection. Let \(j\colon S\nu \to X\) 
be its restriction to \(S\nu\). Note that 
\(j(S\nu)\) is disjoint from \(F\) and that \(j\) is a canonically 
\(\TT\)-equivariantly 
\(\K\)-oriented embedding with trivial normal bundle.
To see this, define 
\[\hat{f}\colon S\nu \times \R \cong U_F \subset X-F, \; \; \; 
\hat{f} (x, \xi, s) \defeq \hat{\varphi} (s\xi ).\] 
The restriction of \(\hat{f}\) to the zero section 
\(S\nu \times \{0\}\) is the embedding \(j\). 

The class in \(\KK_1^\TT(C(F), C_0(X-F))\) 
of the \(\TT\)-equivariant extension 
\[ 0 \longrightarrow C_0(X-F)\longrightarrow 
C(X) \longrightarrow C(F) \longrightarrow 0\]
is equal 
(see \cite{Connes-Skandalis:Longitudinal} Proposition 3.6.; the 
equivariant version goes through in the same way since 
we have a \(\TT\)-equivariant normal bundle) to the class 
of the 
\(\TT\)-equivariant correspondence 
\[S\nu \xleftarrow{\pi_{S\nu}}(S\nu \times \R, \beta_\R) 
\xrightarrow{\hat{f}}X-F\]
where \(\beta_\R\in \K^1_\TT (\R)\) is the Bott class
(for the trivial \(\TT\)-action on \(\R\).) 

Hence the class
\(\partial [V] \in \K^1_\TT (X-F)\) is then represented by the 
smooth \(\TT\)-equivariant correspondence 
\( \pnt \leftarrow (S\nu, \pi_{S\nu}^*(V) ) \xrightarrow{j} X\setminus F\), 
alternatively, as the class
\[ \hat{f}_! \bigl( \pi_{S\nu}^*(V)\cdot \beta_\R\bigr)\in \K^1_\TT(X-F)\]
of the Thom class of the (trivial) 
normal bundle, pushed forward to \(X-F\) via \(\hat{f}\).

Note that since \(\TT\) acts freely on \(\nu - 0\), the open neighbourhood
\(U_F\) of \(\hat{\varphi}(S\nu)\) may be assumed to meet none of the 
\(X^\gamma\) with \(\gamma\in \TT - \{1\})\).  Hence localizing at \(\gamma\not= 1\)
kills the range of \(\partial_0\), so its range is contained 
in the component of \(\K^1_\TT(X-F)\) over 
\(1\in \TT\). Similarly for \( i=1\). 

For the last statement, 
we know from the general discussion above that 
\(\ker(\partial_{i+1}) \cong \textup{Free}\bigl(\K^i_\TT(X)\bigr)\) as \(\Laur\)-modules, 
which implies the corresponding statement after localization at \(\C^*-\{1\}\). 
 But we have just argued that \(\partial_{i+1}\) induces the zero map after 
localization at \(\C^*-\{1\}\), so that its kernel after localization becomes 
\(\K^i_\TT(F)_f\) (\(f(X) = X-1\)). Hence the free \(\Laur_f\)-modules 
\(\textup{Free}\bigl(\K^i_\TT(X)_f\bigr)\) and \(\K^i_\TT(F)_f\) are 
isomorphic, so have the same rank, and it follows 
that \(\textup{Free}\bigl(\K^i_\TT(X)\bigr)\) and \(\K^i_\TT(F)\) have 
the same rank also, since localizing a free module does not change 
its rank.

\end{proof}

\begin{remark}
We make several remarks about the proof. 
\begin{enumerate}
\item We can describe the maps \(\partial_i'\) more 
precisely. In the proof of 
Corollary \ref{cor:spectra_away_from_one} we 
observed that there is a \(\TT\)-equivariant correspondence
\[S\nu \xleftarrow{\pi_{S\nu}}(S\nu \times \R, \beta_\R) 
\xrightarrow{\hat{f}}X-F.\]
In fact by shrinking the neighbourhood 
\(U_F\) of \(S\nu\) if needed so that it is disjoint from 
\(F\), we can factor \(\hat{f}\) through 
an open embedding 
\( \hat{f}'\colon S\nu \times \R \to U_F-F\) and the 
open embedding \(U_F-F \to X-F\). The first yields a 
class in \(\KK^1_\TT ( C(S\nu), C_0(U_F - F))\) but 
this group maps, using descent, to 
 \(\KK^1_\TT (C(\TT \backslash S\nu), C_0(U_F-F))\) since 
 \(\TT\) acts freely on \(S\nu\) and \(U_F-F\). 
 Now the open embedding 
 \(U_F-F \to X-F\) induces an open embedding of 
 quotient spaces \(\TT \backslash U_F -F \to 
 \TT \backslash X-F\) and an element 
 \(j! \in \KK(C_0(\TT\backslash U_F-F), C_0(\TT\backslash X-F))\). 
 The map \(\partial_i'\) is the composition 
  \begin{multline}
   \K^i_\TT (F) \xrightarrow{\pi_{S\nu}^*}\K^i_\TT(S\nu) 
  \cong \K^i (\TT \backslash S\nu) \xrightarrow{\hat{f}'}
  \K^{i+1}(U_F-F)  \xrightarrow{j!}\K^{i+1}(\TT\backslash X-F).\end{multline}
\item The boundary map \(\partial_0\colon \K^0(F)\otimes \Laur \to \K^1_\TT (X-F)\) 
may be understood as giving an obstruction to extending a 
\(\TT\)-equivariant vector bundle on \(F\) to a \(\TT\)-equivariant 
vector bundle on \(X\): this is possible for a given \([V]\) only if 
\(\partial_0[V]= 0 \), which is if and only if the class
\[ \hat{f}_! \bigl( \pi_{S\nu}^*(V)\cdot \beta_\R\bigr)\in \K^1(\TT\backslash X-F)\] 
vanishes. 
\end{enumerate}
\end{remark}

%\begin{remark}
%It would be interesting to study further the maps \(\partial_i'\) and the 
%obstruction groups \(\K^i(\TT\backslash X-F)\,/\,\textup{im}(\partial_i')\) 
%but 
%we do not pursue this further now.  
%\end{remark}

\section{The Lefschetz theorem}
\label{sec:lef}
\begin{definition}
\label{def:lef_map}
Let \(X\) be a smooth, compact
 \(\TT\)-manifold. Let 
 \begin{itemize}
 \item \(D\in \KK^\TT_0(C_0(\Tvert X), \C)\) be the class of the \(\TT\)-equivariant Dirac operator on the almost-complex manifold \(\Tvert X\). 
 \item \(\Theta \in \KK^\TT_0\bigl( C_0(X), C_0(X\times \Tvert X)\bigr) \) the class of the \(\TT\)-equivariant \(\K\)-oriented embedding \(\rho \colon X\to X\times \Tvert X\), 
 \(\rho (x) \defeq 
 \bigl(x, (x,0)\bigr)\). 
 \item \(s\) be the proper \(\TT\)-map \( \Tvert X\to X\times \Tvert X\), 
 \(s (x,\xi) \defeq \bigl( (x,\xi), x \bigr)\).   
  \end{itemize}

 Then the \emph{Lefschetz map} (see \cite{Emerson-Meyer:Dualities}) 
\[ \Lef\colon \KK^\TT_*\bigl(C(X), C(X)\bigr) \to \KK^\TT_*(C(X), \C\bigr)\]
is the composition
\begin{multline}
\label{eq:lef_map}
\KK^\TT_*\bigl(C(X), C(X)\bigr)
\xrightarrow{\otimes_\C 1_{\Tvert X}}
\KK^\TT_*( C_0(X\times \Tvert X), C_0(X\times \Tvert X) ) 
\\ \xrightarrow{ s^*}
\KK^\TT_*\bigl(C_0(X\times \Tvert X), C_0(\Tvert X)\bigr)
\\ \xrightarrow{\otimes_{C_0(\Tvert X)} D}
\KK^\TT_*\bigl( C_0(X\times \Tvert X), \C)
\xrightarrow{\Theta\otimes_{C_0(X\times \Tvert X)}} \KK^\TT_*(C(X), \C)
 \end{multline}
\end{definition}

Thus the Lefschetz map associates to an equivariant 
morphism \(X\to X\) in \(\KK^\TT\), an equivariant 
\(\K\)-homology class for \(X\). Such a class has an index in 
\(\Rep (\TT) \cong \Laur\). 

\begin{definition}
\label{def:lindex}
The \emph{Lefschetz index} \(\Lindex (\Lambda)\), where 
\(f\in \KK^\TT_*(C(X), C(X))\) is the \(\TT\)-equivariant index 
\[ \Lindex (\Lambda) \defeq (\pnt)_*\, \Lef (\Lambda) \in \Rep (\TT) \cong \Laur,\]
where \(\pnt \colon X \to \pnt\) is the map to a point. 
\end{definition}

In \cite{Emerson-Meyer:Correspondences}
and \cite{Emerson-Meyer:Wrong_way} we proved that 
\(\TT\)-equivariant correspondences are cycles for a  
bivariant homology theory isomorphic to 
\(\KK^\TT\), with some restrictions on its arguments (\emph{e.g.} to 
compact smooth \(\TT\)-manifolds.) 

 Hence both the domain and co-domain of the 
Lefschetz map can be described in terms of equivalence 
classes of correspondences; since 
we have defined the Lefschetz map itself in terms of 
correspondences, the Lefschetz map can be described 
in purely geometric terms. We give a brief summary. 

Suppose the following data is given (see 
the original reference \cite{Connes-Skandalis:Longitudinal}), or 
\cite{Emerson-Meyer:Correspondences}.)

\begin{itemize}
\item \(M\) is a smooth \(\TT\)-manifold (not necessarily compact). 
\item \(b\colon M \to X\) is a smooth \(\TT\)-map (not necessarily proper). 
\item \(\xi \in \textup{RK}^*_{\TT, X}(M)\) is an equivariant \(\K\)-theory class with 
compact support along the fibres of \(b\). 
\item \(f\colon M \to X\) is a \(\TT\)-equivariant smooth \(\K\)-oriented map. 
\end{itemize}
This data is sometimes summarized by a diagram 
\( X \xleftarrow{b} (M, \xi) \xrightarrow{f} X\). The quadruple 
\( (M, b, f, \xi)\) is a 
\emph{\(\TT\)-equivariant correspondence} from \(X\) to \(X\). 

It is convenient to assume that the 
correspondence -- denote it \(\Lambda\) --
also satisfies 
\begin{itemize}
\item \(f\colon M \to X\) is a submersion. 
\item The map \(X \to X\times X\), \(x \mapsto \bigl( f(x), b(x)\bigr)\) is 
transverse to the diagonal \(X\to X\times X\). 
\end{itemize}

These conditions imply that the \emph{coincidence space}
\[ \Fixed \defeq \{ x\in M \mid f(x) = g(x)\] has the structure of a 
smooth, equivariantly \(\K\)-oriented
 \(\TT\)-manifold (probably disconnected, but with only finitely many 
 connected components, but each of the same dimension.)
 
  Clearly it comes with a 
 map \(b|_\Fixed\colon \Fixed \to X\), so we obtain a Baum-Douglas 
 cycle \((\Fixed, b|_\Fixed, \xi|_\Fixed)\) 
 for \(X\) by restricting \(\xi\) to \(\Fixed\subset M\). 
 
 To a correspondence is associated a class, which by abuse of 
 notation we also denote by \(\Lambda\), in \(\KK^\TT_*(C(X), C(X))\).  
 Here \(* = \dim(M) - \dim (X) + \dim (\xi)\). See 
 \cite{Emerson-Meyer:Correspondences} for the details.

 The following is a straightforward manipulation with correspondences. 
 
 \begin{proposition}
\label{prop:lef_geometrically}
If \(\Lambda \in \KK^\TT_*(C(X), C(X))\) is represented by the \(\TT\)-equivariant 
correspondence in general position in the sense 
described above, then \(\Lef (\Lambda)\) is represented by the 
Baum-Douglas cycle \((\Fixed, b|_\Fixed, \xi|_\Fixed)\) for \(X\). 
In particular, \[\Lindex (\Lambda
) = \ind_\TT (D_\Fixed\cdot \xi|_{\Fixed})\in \Rep (\TT) \cong \Laur\]
holds; that is, the Lefschetz index of \(\Lambda\) equals  
the \(\TT\)-index of the \(\TT\)-equivariant Dirac operator on 
the coincidence manifold \(\Fixed\), twisted by \(\xi|_\Fixed\). 
\end{proposition}

We will not prove this proposition; the proof can be found in 
\cite{Emerson-Meyer:Lefschetz} or the reader reasonably familiar 
with correspondences can prove it himself.

We aim to prove that \(\Lindex (\Lambda) 
= \trace_{\Laur} (\Lambda_*)\) for 
where 
\(\Lambda_*\colon \K^*_\TT (X) \to \K^*_\TT (X)\) is the action of 
\(\Lambda\) on equivariant \(\K\)-theory; note that \(\Lambda_*\) is a \(\Laur\)-module 
map. Proving this statement has nothing to do with correspondences; it 
depends only on formal properties of \(\KK^\TT\). 

The
 result 
provides a homological interpretation of the Lefschetz index 
along the lines of the classical theorem. 

By the \emph{trace} we mean the following. Firstly, since  
\(X\) is a smooth compact manifold, 
\(\K^*_\TT (X)\) is a finitely generated \(\Laur\)-module. 
Therefore (in each dimension \(* = 0,1\) 
it decomposes into a free part and a torsion part. 
Any \(\Laur\)-module self-map of \(\K^*_\TT (X)\) of even degree will 
induce a grading-preserving map on \(\K\)-theory. We will define the 
trace of such a map to be the differences of the \(\Laur\)-valued 
module traces on 
\(\K^0_\TT(X)\) and \(\K^1_\TT(X)\). To define these individually, 
consider any \(\Laur\) module, which we write as  
\(M = T\oplus \Laur^k\) where \(T\) is torsion. Any self \(\Laur\)-module 
map of \(M\)
 sends \(T\) to itself and hence 
has an upper-triangular form 
\(L = \left[ \begin{matrix} A & B \\ 0 & C\end{matrix}\right]\)
and we let \(\trace_{\Laur} (L) \defeq \trace_{\Laur} (C)\). This is 
uniquely defined. 

A \(\Laur\)-module self-map of \(\K^*_\TT (X)\) with odd degree 
will have trace zero, by definition. 

\begin{theorem}(Lefschetz theorem in \(\KK^\TT\)). 
\label{thm:lef}
Let \(X\) be a compact smooth \(\TT\)-manifold and 
\(\Lambda \in \KK^\TT_*(C(X), C(X))\). Then 
\(\Lindex (\Lambda) 
= \trace_{\Laur} (\Lambda_*)\). 
\end{theorem}

Before proceeding, note that since \(\Lef\) (and \(\Lindex\)) are 
both defined by basic \(\KK^\TT\)-operations, both maps are
 compatible 
in the obvious sense with localization. For any 
\(A\) and \(B\) and any \(\alpha \in \KK^\TT_*(A,B)\), and any 
\(f\in \Laur\), 
denote by \(\alpha_f\in \KK^\TT_*(A,B)_f\) the image of 
\(f\) under localization at \(U_f\). Then compatibility means that 
the diagram
\begin{equation}
\label{eq:lef_and_loc}
\xymatrix{ 
\KK^\TT_*(C(X), C(X)) \ar[d] \ar[r]^{\Lef} &\KK^\TT_*(C(X), \C) \ar[d] \ar[r]^{\ind_\TT}
& \Laur \ar[d] \\
\KK^\TT_*(C(X), C(X))_f  \ar[r]^{\Lef} & \KK^\TT_*(C(X), \C)_f
\ar[r]^{\ind_\TT} & \Laur_f}
\end{equation}
commutes, where the lower row is the `localized' Lefschetz index map, 
defined using Kasparov products as on the top row, except with the 
localized classes \(D_f, \Theta_f\) and so on.

Neither the first nor second vertical map need be injective, of 
course, but the third vertical map is injective because 
\(\Laur\) is an integral domain. The diagram says that 
\( \Lindex (\Lambda)_f = \Lindex_f(\Lambda_f)\) where 
\(\Lindex_f\) is the Lefschetz map in localized \(\KK^\TT\). 

We define the localized module trace 
\[\trace_{\Laur_f}\colon \End_{\Laur_f} (\K^*_\TT (X)_f) \to \Laur_f\] 
as with the non-localized version. Note that localization 
of a \(\Laur\)-module respects the decomposition into 
its torsion and free parts, so that 
\begin{equation}
\label{eq:alg_lef_loc}
 \trace_{\Laur_f} (L_f) = \left[ \trace_{\Laur} (L)\right]_f
 \end{equation} is clear, for any 
\(\Laur\)-module self-map of \(\K^*_\TT (X)\). 

It will be sufficient to prove the following apparently 
weaker version of Theorem 
\ref{thm:lef}. 

\begin{lemma}
\label{lem:localized_lef}
Let \(\Omega\) be as in Theorem \ref{thm:restriction_to_stationary} 
and \(f\in \Laur\) vanish on \(\Omega\). Then the Lefschetz theorem 
for \(X\) holds in 
\(\KK^\TT_f\). That is, 
\[ \Lindex_f (\Lambda_f) = \trace_{\Laur_f} \bigl( (\Lambda_f)_*\bigr)\]
for any \(\Lambda \in \KK^\TT_*(C(X), C(X))\). 
\end{lemma}

Lemma \ref{lem:localized_lef} implies Theorem 
\ref{thm:lef} because combining the diagram \eqref{eq:lef_and_loc}
and its algebraic analogue \eqref{eq:alg_lef_loc} gives 
\begin{multline}
 \Lindex (\Lambda)_f = 
\Lindex_f(\Lambda_f) = \trace_{\Laur_f}\bigl(\Lambda_f)_*\bigr)
\\= \left[\trace_{\Laur} ( \Lambda)\right]_f \in \Laur_f.
\end{multline}
By injectivity of \(\Laur \to \Laur_f\), 
it follows that \(\Lindex (\Lambda) = \trace_{\Laur} ( \Lambda_*)\), 
yielding Theorem \ref{thm:lef}.

To prove Lemma \ref{lem:localized_lef} it is useful to use a slightly different
formalism for the Lefschetz \emph{indices} \(\Lindex (\, \cdot \,)\). This 
formalism  
is more general in the sense that it applies to noncommutative 
\(\TT\)-C*-algebras as well, provided they have duals. (The Lefschetz 
\emph{map} of Definition \ref{def:lef_map} exists in more generality than 
we have suggested, but does not work for noncommutative algebras 
because of the implicit use of the `diagonal map' \(X \to X\times \Tvert X\).)

As above, \(s \colon \Tvert X\to X\times \Tvert X\) is the obvious section. 
Let \(\Sigma \colon X\times \Tvert X \to \Tvert X \times X\) be the flip. 
Set 
\begin{itemize}
\item \(\Delta \defeq \Sigma^*s_*(D)\in \KK^\TT( C_0(\Tvert X\times X), \C)\), 
\item \( \underline{\Dudelta} \defeq (\pnt)_*(\Theta) \in \KK^\TT(\C, C_0(X\times \Tvert X))\), 
\end{itemize}

We denote \(A \defeq C(X)\) and \(B \defeq C_0(\Tvert X)\). 

It is easily checked that \(\Delta\) and \(\underline{\Dudelta}\) satisfy the 
`zig-zag equations' 
\begin{equation}
\label{eq:zig_zag}
 \bigl( \underline{\Dudelta} \otimes_\C 1_A\bigr) \otimes_{A\otimes B\otimes A} \bigl( 1_A\otimes \Delta\bigr) = 1_A,\; \; 
\bigl( 1_B\otimes_\C \underline{\Dudelta} ) \otimes_{B\otimes A\otimes B} \bigl( \Delta\otimes_\C 1_B\bigr)= 1_B\end{equation}
and it follows that the map 
\[ \KK^\TT_*(D_1, D_2\otimes B) \to \KK^\TT_*(D_1\otimes , D_2),\; 
x\mapsto (x\otimes 1_A)\otimes_{B\otimes A} \Delta\]
is an isomorphism for every \(D_1, D_2\) (\emph{c.f.} 
the briefly stated Theorem \ref{thm:duality}). The inverse map
is defined similarly, using \(\underline{\Dudelta}\). This is the kind of 
noncommutative Poincar\'e duality studied by the author in 
several 
papers, \emph{e.g.}
\cite{Emerson:Duality_hyperbolic} and \cite{Emerson:Lefschetz_numbers}. 
See also the discussion in \cite{Connes:NCG},
and the survey \cite{Rosenberg:duality}. 

Set \(\Dudelta \defeq \Sigma_*(\underline{\Dudelta})\).

\begin{lemma}
\label{lem:noncommutative_lef}
In the above notation: for any \(\Lambda \in \KK^\TT_*(A, A)\defeq \KK^\TT_*(C(X), C(X))\), 
\begin{equation}
\label{eq:lindex_noncommutative}
\Lindex(\Lambda) 
= \bigl( \Dudelta\otimes_{B\otimes A} (1_B\otimes \Lambda)\bigr)\otimes_{B\otimes A} \Delta\in \KK^\TT_*(\C, \C) \cong \Laur
\end{equation}
Similarly after localization. 
\end{lemma}

\begin{proof}
Using the definitions 
\begin{multline}
\Lindex (\Lambda) \defeq 
(\pnt)_*\bigl(\Lef (\Lambda) \bigr)
\\ = (\pnt)_*(\Theta)\otimes_{C_0(X\times \Tvert X)} (\Lambda\otimes_\C 1_{C_0(\Tvert X)} ) 
\otimes_{C_0(X\times \Tvert X)} [s^*]\otimes_{C_0(\Tvert X)} D
\\ = \underline{\Dudelta}\otimes_{C_0(X\times \Tvert X)} (\Lambda\otimes 1_{C_0(\Tvert X)})
\otimes_{C_0(X\times \Tvert X)} \Sigma^*(\Delta).
\end{multline}
where \([s^*]\in \KK^\TT (C_0(X\times \Tvert X), C_0(\Tvert X) ) \) is the class of
\(s\). Carrying the flip across yields 
\begin{equation}
= \Dudelta \otimes_{C_0(\Tvert X\times )} (1_{C_0(\Tvert X)}\otimes_\C \Lambda) \otimes_{C_0(\Tvert X\times X)} \Delta 
\end{equation}
as required. 
\end{proof}

In particular, using the right hand side of 
\eqref{eq:lindex_noncommutative}, we can define
 the Lefschetz index of a morphism 
\(\Lambda \in \KK^\TT_*(A,A)\) for any \(\TT\)-C*-algebra 
\(A\) for which there exists a triple \( (B, \Delta, \underline{\Dudelta})\) satisfying 
\eqref{eq:zig_zag}. We call such \(A\) \emph{dualizable}. 

The author believes  
that \(A\) dualizable implies \(\K^\TT_*(A)\) is a finitely generated 
\(\Laur\)-module (see \cite{Emerson-Meyer:Dualities} for the 
non-equivariant proof) but does not have a reference. We 
are not interested in proving this here, since the \(A\) 
we consider obviously have finitely generated equivariant 
\(\K\)-theory.

Suppose for such \(A\) there exists a C*-algebra \(A'\) and a 
\(\KK^\TT\)-equivalence \(\alpha \in \KK^\TT(A,A')\). In this 
case, \(A'\) is also dualizable using \(B'\defeq B\), 
\begin{equation}
\label{eq:isomorphic_dual}
\Delta'\defeq 
(1_B\otimes_\C \alpha^{-1})\otimes_{B\otimes A} \Delta \in \KK^\TT(B' \otimes A', \C)
\end{equation} and 
\begin{equation}
\label{eq:isomorphic_dual_two}
\underline{\Dudelta}' \defeq \underline{\Dudelta}\otimes_{A\otimes B} (\alpha\otimes 1_B)\in  
\KK^\TT(\C, A'\otimes B').
\end{equation}
Conjugation by \(\alpha\) gives an isomorphism 
\(\KK^\TT_*(A,A) \cong \KK^\TT(A',A')\) and it is easy to check that 

\begin{lemma}
\label{lem:functoriality}
\begin{equation}
\Lindex (\Lambda) = \Lindex(\alpha \otimes_{A'} \Lambda \otimes_{A'} \alpha^{-1})
\end{equation}
for any \(\Lambda \in \KK_*(A',A')\), and where the left-hand-side of this 
equation is defined using the dual \( (B',\Delta', \underline{\Dudelta}')\) and 
the right-hand-side using \( ( B, \Delta, \underline{\Dudelta})\). 

\end{lemma}
This is of course what is to be expected if \(\Lindex\) is to 
agree with a 
\(\Laur\)-valued trace: the statement
\[ \trace_{\Laur} (\Lambda_*)
= \trace_{\Laur} (\alpha_*^{-1} \circ \Lambda_*\circ \alpha_*)\]
with \(\Lambda_*\colon \K_*^\TT (A') \to \K^*_\TT (A')\), 
\(\alpha_*\colon \K_*^\TT (A) \to \KK^\TT_*(A')\) the 
module maps induced by \(\Lambda\) and \(\alpha\), 
is obvious. 

Lemma \ref{lem:functoriality} also proves the independence 
of 
\[ \Lindex \colon \KK^\TT_*(A,A) \to \Laur\]
of the choice of dual \( (B, \Delta, \Dudelta)\), since any 
two duals for a fixed \(\TT\)-C*-algebra \(A\) are related by a 
self-\(\KK^\TT\)-equivalence of \(B\) as in \eqref{eq:isomorphic_dual}
and \eqref{eq:isomorphic_dual_two}.

This discussion has its obvious analogue in the 
localized category \(\KK^\TT_f\) (for any \(f\in \Laur\)). That is, 
we can speak of a C*-algebra \(A\) being 
\emph{dualizable} in \(\KK^\TT_f\), we may define the 
Lefschetz index map \(\Lindex_f\colon \KK^\TT_*(A,A)_f \to 
\Laur_f\), and so on, \emph{c.f.} the discussion around 
\eqref{eq:lef_and_loc} regarding the Lefschetz \emph{map} 
for \(A= C(X)\).

We now 
return to the case where \(A = C(X)\) for a smooth, compact  \(\TT\)-manifold 
\(X\).  Let \(f\) be as in Theorem \ref{thm:restriction_to_stationary}.
Thus \(\rho_f\in \KK^\TT(C(X), C(F))_f\) is a \(\KK^\TT_f\)-equivalence. 
Hence by the analogue in \(\KK^\TT_f\) of Lemma \ref{lem:functoriality}, 
\[ \Lindex_f (\Lambda_f) = \Lindex_f ( \rho_f^{-1}\otimes_{C(X)} \Lambda_f
\otimes_{C(F)} \rho_f).\] Note that 
\(\Lindex_f\) is defined for the stationary set \(F\) because already 
\[\Lindex \colon \KK^\TT_*(C(F), C(F)) \to \KK^\TT_*(C(F), \C)\]
is defined, because the stationary set \(F\) is a smooth 
\(\TT\)-manifold (with the trivial action), and hence has a dual.

 Our goal at this stage is therefore to prove that 
\begin{equation}
\label{eq:reduced_to}
\Lindex_f (\mu) = \trace_{\Laur_f} (\mu_*)
\end{equation}
for any \(\mu \in \KK^\TT(C(F), C(F))_f\). This will prove 
Lemma \ref{lem:localized_lef} and hence Theorem \ref{thm:lef}. 
But since \(F\) is a trivial \(\TT\)-space, we can 
prove even the stronger statement 
\begin{equation}
\label{eq:trivial_lef}
\Lindex (\mu) = \trace_{\Laur} (\mu_*)
\end{equation}
In fact this is simply a computation with bilinear forms, and 
applies to general groups. 

%In fact at this stage, since we have only trivial 
%\(\TT\)-spaces \(F\) and \(\Tvert F\), the Lefschetz theorem is 
%proved exactly like its non-equivariant 
%counterpart (worked out in \cite{Emerson:Lefschetz_numbers}). 

\begin{lemma}
\label{lem:lef_for_trivial_actions}
Let \(G\) be a compact group, let \(A\) be a trivial 
\(G\)-C*-algebra and \(B\) a \(G\)-C*-algebra. Assume that
as a C*-algebra, \(A\) is in the boostrap category \(\mathcal{N}\). 
Finally, assume that \(B\) and \(A\) are Poincar\'e dual, \emph{i.e} 
that there exist classes 
\( \Delta \in \KK^G_0(B\otimes A, \C)\) and 
\(\underline{\Dudelta} \in\KK_0^G (\C, A\otimes B)\) such that 
\eqref{eq:zig_zag} are satisfied. Let \(\Lambda \in 
\KK_*(A,A)\) and \(\Dudelta \defeq \Sigma_*(\Dudelta) \in \KK^G_0(\C, B\otimes A)\).
 Then 
 \[ \bigl( \Dudelta \otimes_{B\otimes A} (1_B\otimes \Lambda)\bigr)
 \otimes_{B\otimes A} \Delta = \mathrm{trace}_s (\Lambda_*)\]
 holds, where the trace is that of the module map 
 induced by \(\Lambda\) on the \emph{free, finitely generated} 
 \(\Rep (G)\)-module \(\K_*^G (A) \cong \K_*(A)\otimes_\C \Rep (G)\).  
 
In particular, Lemma \ref{lem:localized_lef}, and hence 
Theorem \ref{thm:lef} and (hence) all of 
its localized 
analogues (in particular \eqref{eq:reduced_to}) 
 hold for trivial compact \(\TT\)-manifolds \(X\). 
\end{lemma}

\begin{proof}
Since \(A\) is a trivial \(G\)-C*-algebra, 
\(\KK^G_*(\C, B\otimes A) \cong \KK_*(\C, A\otimes \, B\rtimes G)\).
The assumed equivariant duality implies non-equivariant duality and  
this implies (see \cite{Emerson-Meyer:Dualities} that 
\(\K_*(A)\) is finite-dimensional. By the Green-Julg theorem  
\(\KK^G_*(\C, B\otimes A) \cong \KK_*(\C, B\rtimes G\, \otimes A)\), 
and by the (non-equivariant) 
K\"unneth theorem 
(\cite{Black} Theorem 23.1.3) this is 
\(\cong \KK_*(\C, B\rtimes G) \otimes \KK_*(\C, A) 
\cong \KK_*^G(\C, B)\otimes \KK^G_*(\C, A)\);
where the tensor product is in the 
category of \(\Rep (G)\)-modules.

 Thus, external product 
\[ \KK^G_*(\C, A) \otimes \KK^G_*(\C, B) \to \KK^G_*(\C, A\otimes B)\]
is an isomorphism; for emphasis, the tensor product on the 
left-hand-side is in the category of \(\Rep (G)\)-modules. 

We may find a finite basis \(\{ y^\epsilon_i\}\) for 
\(\K_*^G(A)\) as an \(\Rep (G)\)-module with \(y^\epsilon_i 
\in \K_\epsilon^G(A)\), and there exist 
\(x^\epsilon_i \in \K_\epsilon^G(B)\) such that 
\begin{equation}
\label{eq:Dudelta_expanded}
\underline{\Dudelta} = \sum_i y^\epsilon_i\otimes_\C x^\epsilon_i \in \K_0^G(A\otimes B).
\end{equation}
We have assumed (\eqref{eq:zig_zag}) that 
\begin{equation}
\label{eq:dual_one}
(\underline{\Dudelta}\otimes_\C 1_A) \otimes_{A\otimes B\otimes A} (1_A\otimes \Delta) = 1_A\in \KK_0(A,A).
\end{equation}
Applying the functor from the category \(\KK^G\) to the category of 
\(\Z/2\)-graded \(\Rep (G)\)-modules, we get that 
\[ y =  y\otimes_A \bigl( ( \underline{\Dudelta}\otimes 1_A) 
\otimes_{A\otimes B\otimes A} 
(1_A\otimes \Delta) \bigr)\]
for all \(y \in \K_*(A)\). 
Expanding the right-hand-side using \eqref{eq:Dudelta_expanded}
yields 
\begin{multline}
y = \sum_{i, \epsilon} \bigl( y^\epsilon_i\otimes_\C x_i^\epsilon\otimes_\C y\bigr)\otimes_{A\otimes B\otimes A} (1_A\otimes \Delta)\\
= \sum_{i, \epsilon} y^\epsilon_i \otimes_\C \bigl( (x_i^\epsilon\otimes_\C y)\otimes_{B\otimes A}\Delta\bigr) = \sum_{i, \epsilon} L^\epsilon_i (y)y^\epsilon_i
\end{multline}
where, as indicated, \(L^\epsilon_i (y) = (x^\epsilon_i \otimes_\C y)\otimes_{B\otimes A}\Delta\in \Rep (G)\). 
Since the \( y^\epsilon_i\) form a basis, we deduce by setting 
\(y = y^\gamma_j\), that 
\begin{equation}
\label{eq:deltaij}
L^\epsilon_i (y^\gamma_j) = \delta_{\epsilon, \gamma}\delta_{i,j}.
\end{equation}
Now let \(\Lambda \in \KK_0(A,A)\) (similar computations apply to 
odd morphisms.) We can write 
\begin{equation}
\Lambda_*(y^\epsilon_i) = \sum_j \lambda_{ij}^\epsilon y^\epsilon_j.
\end{equation}
Since the external product in \(\KK^G\) is graded commutative, 
\[ \Dudelta \defeq \Sigma_*(\underline{\Dudelta}) = \sum_{i,\epsilon} (-1)^\epsilon 
x^\epsilon_i\otimes_\C y^\epsilon_j.
\]
We get, therefore,  
\begin{multline}
\bigl( \Dudelta \otimes_{B\otimes A} (1_B\otimes \Lambda) \bigr)\otimes_{B\otimes A} \Delta 
\\= \sum_{i, \epsilon} (-1)^\epsilon \bigl(x_i^\epsilon\otimes_\C y^\epsilon_i \bigr) \otimes_{B\otimes A} (1_B\otimes_\C \Lambda) \otimes_{B\otimes A} \Delta \\
= \sum_{i, j,\epsilon} (-1)^\epsilon \lambda^\epsilon_{ij} (x^\epsilon_i \otimes_\C y^\epsilon_j)\otimes_{B\otimes A} \Delta
= \sum_{i, \epsilon} \lambda_{ii}^\epsilon 
\end{multline}
where the last step is using \eqref{eq:deltaij}. This gives the graded trace of 
\(\Lambda_*\) acting on the free \(\Rep (G)\)-module \(\K^G_*(A)\) as required. 

The last statement follows from setting \(A = C(X)\) as in the discussion 
around \eqref{eq:zig_zag}. 

\end{proof}

We close with a brief discussion of equivariant Euler numbers, 
in order to illustrate the Lefschetz theorem. 

\begin{remark}
The case of 
Euler numbers is the case where \(\Lambda \) is a `twist' of the 
identity correspondence, thus \(\Lambda\) has the form 
\( X\xleftarrow{\Id} (X, \xi) \xrightarrow{\Id} X\) where 
\(\xi \in \K^*_\TT (X)\). 
\end{remark}

We first make a general observation about the Lefschetz map.

%Note that \emph{both the domain and range of the Lefschetz 
%map \eqref{eq:lef_map} are modules over 
%\(\K^*_\TT (X)\)}. Moreover, the following is not hard to 
%check. 

\begin{lemma}
For any \(\TT\)-C*-algebra \(A\) and any \(\TT\)-space 
\(X\),  \(\KK^\TT_*(C(X), A)\) is a module over 
\(\K^*_\TT (X)\). This module structure is `natural' with 
respect to \(A\).

Moreover, 
\(\Lef\colon \KK^\TT_*(C(X), C(X)) \to \KK_*^\TT (C(X), \C)\) 
is a \(\K^*_\TT (X)\)-module homomorphism. 
\end{lemma}

The \(\K^*_\TT (X)\)-module structure on \(\K_*^\TT (X)\) corresponds
to the process of twisting an elliptic operator by a vector bundle.
Furthermore, it follows from the axiomatic definition of the 
Kasparov product that the Kasparov 
pairing \(\K^*_\TT (X)\times \K_*^\TT (X) = 
\KK^\TT_*(\C, C(X)) \times \KK^\TT_*(C(X), \C)\) 
maps 
\( (\xi, a)\) to 
\(\pnt_*(a\cdot \xi)\), 
where the dot is the module structure, \(\pnt\colon X \to \pnt\) is the 
map from \(X\) to a point. 

 We therefore have 
\[ \langle \xi, \Lef (\Lambda) \rangle = \Lindex (\Lambda \cdot \xi)\in \Laur\]
for any \(\Lambda \in \KK^\TT_* (C(X),C(X))\) and 
\(\xi \in \K^*_\TT (X)\), and, roughly, if we can realize \(\Lef (\Lambda)\) as 
the class of a suitable elliptic operator, then this can be interpreted
as the \(\TT\)-index of that operator twisted by \(\xi\).  The module 
structure can also be easily described explicitly in topological 
terms, using correspondences.

The point is that the action of 
\(\Lambda \cdot \xi\in \KK^\TT_*(C(X), C(X))\) on \(\K^*_\TT (X)\) 
is clearly 
the composition 
\[ \K^*_\TT (X) \xrightarrow{\Lambda_*}\K^*_\TT (X) 
\xrightarrow{\lambda_\xi } \K^*_\TT (X)\]
where the map denoted \(\lambda_\xi \) is \emph{ring multiplication by \(\xi\)}; 
this is clearly a \(\Laur\)-module map. Therefore we 
get a refinement of Theorem \ref{thm:lef} involving the 
twisted Lefschetz numbers \(\Lindex (\Lambda \cdot \xi)
=\langle \xi, \Lef (\Lambda) \rangle\). 

\begin{proposition}
\label{prop:twisted_lef}
In the above notation, 
\[\langle \xi, \Lef (\Lambda) \rangle = \trace_{\Laur}( \Lambda_*\circ \lambda_\xi)\in \Laur\]
for any \(\Lambda \in \KK^\TT_*(C(X), C(X))\) and \(\xi \in \K^*_\TT (X)\), where \(\lambda_\xi\) 
is the \(\Rep (\TT)\)-module homomorphism of ring multiplication by \(\xi\). 
\end{proposition}

We call the elements \( e_X(\xi) \defeq \Lindex (\Id\cdot \xi)\) for 
\(\xi \in \K^*_\TT (X)\), the \emph{twisted \(\TT\)-equivariant Euler numbers of 
\(X\)}.  Note that \(e_X(\xi) = 0\) if 
 \(\xi\) is an odd \(K\)-class. 

We may interpret the Euler numbers in two different ways, given the 
above discussion: 

\begin{itemize}
\item \(e_X(\xi)\) is the \(\TT\)-equivariant analytic index of the de Rham operator on \(X\) twisted by \(\xi\). 
\item \( e_X(\xi)\) is the module trace \(\trace_{\Laur} (L_\xi)\) of ring multiplication by \(\xi\) on \(\K^*_\TT (X)\). 
\end{itemize}

The first statement follows from the computation in 
\cite{Emerson-Meyer:Equi_Lefschetz}, which proves the 
much stronger statement that 
\(\Lef (\Id) = [D_{\textup{dR}}]\in \K_0^G(X)\), where 
\(D_{\textup{dR}}\) is the de Rham (or `Euler') operator on \(X\) 
and \(G\) is \emph{any} locally compact group acting 
properly and smoothly on \(X\). For further information on the 
class of the de Rham operator and related issues, see 
\cite{Emerson-Meyer:Euler} and \cite{Emerson-Meyer:Equi_Lefschetz}, 
and the paper of Rosenberg and L\"uck \cite{Rosenberg-Lueck:Lefschetz} and 
of Rosenberg \cite{Rosenberg:Euler}. 

To compute the invariants in the first interpretation, let
\(g\in \TT\) generate the circle topologically, 
so that \(\textup{Fix} (g) = F\).  Since \(g\colon X \to X\) is 
\(\TT\)-equivariantly homotopic to the identity, 
\(\Lef (\Id) = \Lef ([g^*])\). Now the computation of the 
Lefschetz map (for ordinary smooth self-maps) in 
\cite{Emerson-Meyer:Equi_Lefschetz} yields  
\[ \Lef ([g^*]) = (i_F)_* ([D_{\textup{dR}}^F])\]
where \(D_{\textup{dR}}^F\) is the de Rham operator 
on \(F\), \([D_{\textup{dR}}^F]\) its class in 
\(\KK^\TT_0(C(F), \C)\), and 
\(i_F\colon F \to X\) is the inclusion map. (The sign data 
in \cite{Emerson-Meyer:Equi_Lefschetz} vanishes because 
\(g\) is an isometry, which implies that the vector bundle 
map \(\Id - Dg\) on the \(\TT\)-equivariant normal bundle 
to \(F\) is homotopic to the identity bundle map.)

Thus, we 
see that  
\( e_X(\xi) = e_F( \xi |_F )\) where \(e_F(\xi |_F)\) denotes 
the equivariant Lefschetz number of the restriction of 
\(\xi\) to the smooth (trivial) \(\TT\)-space \(F\). By another 
application of the Lefschetz theorem, this time for the 
trivial \(\TT\)-space \(F\), yields that 
this equals the \(\TT\)-index of the de Rham operator 
on \(F\) twisted by \(\xi |_F\). 

Since \(F\) is \(\TT\)-fixed pointwise, we can further simplify this 
answer. Assume first that \(F\) is connected. The bundle 
\(E|_F\) can be diagonalized into 
eigenspaces for the \(\TT\)-action, \(E|_F \cong \oplus_\lambda  E_\lambda\) 
where \(\TT\) acts on \(E_\lambda\) by the character \(f_\lambda\), some 
\(f_\lambda\in \Laur\). Let \(\xi_\lambda = [E_\lambda]\in \K^0(F)\). 
We see then that  
\[ e_F (\xi |_F) =  \sum_\lambda e_F^{\textup{non-equ.}}(\xi_\lambda)f_\lambda\]
where in this formula \(e_F^{\textup{non-equ.}}\) are the twisted, 
\emph{non-equivariant} 
Euler numbers for the stationary manifold \(F\). 

Non-equivariant Euler numbers are straightforward to 
compute. The index of the de Rham operator on a 
connected compact manifold \(P\), twisted by \(\xi \in \K^0(P)\), 
is simply \(\chi (P)\dim (\xi)\in \Z\), where 
\(\chi\) is the numerical Euler characteristic.

We conclude that 
\[ e_X([E]) = \chi(F) \sum_\lambda \dim_\C (E_\lambda)\,  f_\lambda .\]
If \(F\) has components \(\{P\}\) then this formula becomes
\[ \trace_{\Laur}(\lambda_\xi) = 
e_X(\xi) = \sum_P \chi(P) 
\sum_\lambda \dim_\C ( (E|_P)_\lambda)\,f_{\lambda, P}.\]
The 
right-hand-side is by and large easy to compute in specific situations. The 
case of isolated fixed-points is particularly transparent.

\begin{proposition}
Let \(X\) be a smooth compact 
\(\TT\)-manifold with a finite set of 
isolated stationary points. Then for 
any \(\xi\in \K^0_\TT (X)\), 
\[ \trace_{\Laur}(\lambda_\xi) 
= \sum_{P\in F} \xi_P\]
where the \(\xi_P\) are the  
restrictions of \(\xi\in \K^*_\TT (X)\) to the points \(P\), each 
such \(P\) yielding an element \( \xi_P\in \K^*_\TT (P)\cong \Laur\). 
\end{proposition}

The following example illustrates the difference in computing 
the two invariants equated by the Lefschetz theorem. 

\begin{example}
Let \(X = \mathbb{CP}^1\) with the \(\TT\)-action induced by the 
embedding \( \TT \to \mathrm{SU}_2(\C)\subset \mathrm{Aut}(\C^2), \; z\mapsto \left[
\begin{matrix} z & 0 \\ 0 & \bar{z}\end{matrix}\right]\). There are two stationary 
points, with homogeneous coordinates 
\([1,0]\) and \([0,1]\) respectively. Let \(H^*\) be the canonical 
line bundle on \(\mathbb{CP}^1\), it is a \(\TT\)-invariant
 sub-bundle of \(\mathbb{CP}^1\times \C^2\) so has a canonical structure of 
 \(\TT\)-equivariant vector bundle. Restricting \(H^*\) to the 
 stationary points \([1,0]\) and \([0,1]\) yields respectively the
 characters \(X\) and \(X^{-1}\), whence by the 
 Lefschetz theorem 
 \[e_{\mathbb{CP}^1}([H^k]) = 
 \mathrm{trace}_{\Laur} (\lambda_{[H]}^k) = X^k+X^{-k}\in \Laur.\]
 where \(H\) is the dual of \(H^*\).
Computation of the  \(\mathrm{trace}_{\Laur} (\lambda_{[H]}^k)\) 
by homological methods requires computing 
\(\K^*_\TT(\mathbb{CP}^1)\) as both a ring and as a 
\(\Laur\)-module.  By results of Atiyah and others, 
(see Segal's article \cite{Segal} for a beautiful and concise proof) it is generated as 
a commutative unital ring by \(X\) and \([H]\) with the relations that 
\(X\) and \([H]\) are invertible and commute, and satisfy 
\[ ([H] - X)([H]-X^{-1}) = 0.\]
Hence \([H]^2 = (X+X^{-1}) [H] \, +1\).  
This implies that as a \(\Laur\)-module, \(\K^0_\TT (\mathbb{CP}^1)\) 
is generated by the unit \(1\) of the ring, and the element 
\([H]\). This is a free basis, and with respect to it 
 \[\lambda_{[H]} = \left[ \begin{matrix} 0 & 1\\ 1 & X+X^{-1} \end{matrix}\right].\]
The trace is \(X+X^{-1}\). The formula for \(\trace_{\Laur}(\lambda_{[H]}^k)\)
 follows from induction, using 
the relation 
\(\lambda_{[H]}^n = (X+X^{-1}) \lambda_{[H]}^{n-1} +\lambda_{[H]}^{n-2}\), 
which comes from the relation given by the minimal polynomial 
\( \lambda^2- (X+X^{-1})\lambda -1\) of \(\lambda_{[H]}\).
 
\end{example}

%We conclude by emphasizing that the Lefschetz theorem we have stated 
%and proved here, has little to do with the Index Theorem. The 
%Index Theorem identifies the \(\KK\)-classes built in the 
%standard analytic way from Dirac operators, with other, topologically defined
% \(\KK\)-classes (correspondences) obtained by 
%combining Thom isomorphisms and various embeddings. In the papers  
%\cite{Emerson-Meyer:Wrong_way} and \cite{Emerson-Meyer:Correspondences}
%we have formalized this using a purely topological category 
%\(\widehat{\KK}^G\) of equivalence classes of equivariant correspondences, 
%and a  map \(\widehat{\KK}^G \to \KK^G\). In this language, the Index Theorem 
%amounts to the statement that this topologically defined 
%map sends a correspondence to the 
%corresponding analytically defined class. 

%On the other hand, the Lefschetz theorem should be considered as a
%taking place entirely in the topological theory \(\widehat{\KK}^G\).

%Atiyah and Bott prove their Lefschetz theorem 
%by purely analytic methods, in the framework of elliptic complexes. Their 
%theorem and ours intersect, roughly speaking, for the de Rham complex. The 
%latter has a certain homotopy-invariance. The statement in this case is 
%a purely topological one, and our proof is purely topological. 


\begin{thebibliography}{10}

%\bibitem{Abels}
%Herbert Abels, \emph{A universal proper $G$-space.} Math. Z. 159 (1978), no. 2, 143--158.

\bibitem{Artin-Mazur:Zeta} 
Artin, M., Mazur, Barry, \emph{On periodic points}, Ann. Math., \textbf{81} (1), 82--99. 


\bibitem{Atiyah-Bott:Moment}
Atiyah, M.F., Bott, R., 
\emph{The moment map in equivariant 
cohomology}. Topology 23, no.1 (1984), 1--28. 


\bibitem{Atiyah-Segal:Index}
Atiyah, M.F., Segal, G., 
\emph{The index of elliptic operators II}, Ann. Math., 2nd series, 87, no. 3 (1968),
531--545. 


\bibitem{Atiyah-Segal:Euler} 
Atiyah, M.F., Segal, G., \emph{On equivariant Euler characteristics}, 
J.G.P., {\bf 6}, no. 4, 1989, 671--677. 

\bibitem{Baaj-Skandalis:Duality}
Baaj, S., Skandalis, G.
C*-algebres de Hopf et th\'eorie de Kasparov \'equivariante.
K-theory, \textbf{2} (1989), pp. 683--721.



%\bibitem{Bag}
%L. Baggett.
%\emph{A description of the topology on the dual spaces of certain locally compact groups}. 
%Trans. Amer. Math. Soc. {\bf 132} (1968), 175--215. 

\bibitem{Baum-Connes:finite}
Baum, P., Connes, A., \emph{$K$-theory for discrete groups}.  Operator algebras and applications, Vol. 1,  1--20, London Math. Soc. Lecture Note Ser., 135, Cambridge Univ. Press, Cambridge, (1988).

%\bibitem{BCH} P. Baum, A. Connes and N.  Higson, \emph{Classifying
%space for proper actions and $K$-theory of group C*-algebras},
%Contemporary Mathematics, \textbf{167} (1994), pp. 241-291. 
  
\bibitem{Black}
B.  Blackadar. 
\emph{$K$-theory for operator algebras}. Second edition. Mathematical Sciences Research Institute Publications, 5. Cambridge University Press, Cambridge, 1998. xx+300 pp. ISBN: 0-521-63532-2

\bibitem{Rosenberg:duality}
 Brodzki, J., Mathai, V., Rosenberg, J., Szabo, R.J.,
\emph{D-branes, RR-fields and duality on noncommutative manifolds},
 Comm. Math. Phys. \textbf{277}, no. 3 (2008), pp. 643--706. 

\bibitem{Connes:NCG}
 Connes, A.,
\emph{Noncommutative geometry},
Academic Press Inc.,San Diego, CA. (1994),


\bibitem{Connes-Skandalis:Longitudinal}
 Connes, A., Skandalis, G.
 \emph{The longitudinal index theorem for foliations}.
Publ. Res. Inst. Math. Sci., \textbf{20}, no. 6 (1984),
pp. 1139--1183. 

%\bibitem{Ech:vector_space}
%Echterhoff, S., Pfante, O.
%\emph{Equivariant \(\K\)-theory of finite-dimensional real vector spaces}
%M\"unster J. of Math., \textbf{2}, (2009), pp. 65--94. 

\bibitem{Echterhoff-Emerson:Proper}
Echterhoff, S., Emerson, H.
\emph{Structure and \(\K\)-theory for crossed products by 
proper actions}.
In preparation. 

\bibitem{Emerson:Duality_hyperbolic}
Emerson, H. 
\emph{Noncommutative Poincar\'e duality for boundary actions of hyperbolic groups}.
J. Reine Angew. Math. \textbf{564}, (2003), pp. 1-33. 


\bibitem{Emerson:Lefschetz_numbers}
Emerson, H.
\emph{Lefschetz numbers for C*-algebras}.
Can. Math. Bull.
To appear. 
\textup{arxiv} math.KT/07084278. 


\bibitem{Emerson-Meyer:Euler}
 Emerson, H., Meyer, R.,
\emph{Euler characteristics and Gysin sequences for group actions on boundaries}.
Math. Ann. \textbf{334}, no. 4 (2006), pp. 853--904. 

\bibitem{Emerson-Meyer:Equi_Lefschetz}
Emerson, H., Meyer, R.,
 \emph{Equivariant Lefschetz maps for simplicial complexes and smooth manifolds}.
 Math. Ann. \textbf{345}, no. 3 (2009), pp. 599--630.


\bibitem{Emerson-Meyer:Dualities}
 Emerson, H., Meyer, R. \emph{Dualities in equivariant Kasparov theory},
New York J. Math. {\bf 16}, (2010), pp. 245--313.

  
\bibitem{Emerson-Meyer:Lefschetz}
Emerson, H., Meyer, R.
\emph{An equivariant Lefschetz fixed-point formula for correspondences}.
Preprint. 

\bibitem{Emerson-Meyer:Correspondences}
Emerson, H., Meyer, R.
\emph{Bivariant K-theory via correspondences}.
Adv. Math. {\bf 225}, no. 10 (2010), pp. 2883--2919. 

\bibitem{Emerson-Meyer:Wrong_way}
Emerson, H., Meyer, R. 
\emph{Equivariant embeddings and topological index maps}.
Adv. Math. (2010). 

\bibitem{Hartshorne}
Hartshorne, R. 
\emph{Algebraic geometry}.
Graduate Texts in Mathematics, No. 52. Springer-Verlag, New York-Heidelberg (1977).
ISBN: 0-387-90244-9
  
\bibitem{Kas} G. Kasparov.
\emph{Equivariant
$\KK$-theory and the Novikov conjecture}, Invent. Math.
\textbf{91},  147-201 (1988).

\bibitem{Meyer:Homological_II}
Meyer, R. 
\emph{Homological algebra in bivariant \(\textup{K}\)\nobreakdash-theory and other triangulated categories. II}. Tbilisi Mathematical Journal, \textbf{1}, pp.165--210. 

\bibitem{Meyer:Homological_I}
 Meyer, R., Nest, R.
  \emph{Homological algebra in bivariant \(\textup{K}\)\nobreakdash-theory and other triangulated categories. I}. Preprint. \textup{arxiv} math.KT/0702146 (2007). 


\bibitem{Palais:Slices} Richard S. Palais, 
\emph{On the existence of slices for actions of non-compact Lie groups.} Ann. of Math. (2) {\bf 73} (1961) 295--323.

\bibitem{Putnam:Smale}
Putnam, I. 
\emph{C*-algebras from Smale spaces.}
Canad. J. Math {\bf 48} (1996), 175-195.

\bibitem{Putnam:Homology}
Putnam, I.
\emph{A homology theory for Smale spaces: a summary (Research announcement)}
Pre-print
see http://www.math.uvic.ca/faculty/putnam/r/res.html.

\bibitem{Rosenberg:Euler}
Rosenberg, J.
\emph{The \(\K\)-homology class of the Euler characteristic operator is trivial}. 
Proc. Amer. Math. Soc., \textbf{127} (1999), pp. 3467--3474.

\bibitem{Rosenberg-Lueck:Lefschetz}
Rosenberg, J., L\"uck, W.
\emph{The equivariant Lefschetz fixed point theorem for proper cocompact 
\(G\)-manifolds}. 
High dimensional manifold topology. 
World Sci. Publ., River Edge, NJ. (2003), pp. 322--361. 

\bibitem{Rosenberg-Schochet:UCT}
Rosenberg, J., Schochet, C.
\emph{The K\"unneth theorem and the universal coefficient theorem for equivariant $K$-theory and $KK$-theory}. 
Mem. Amer. Math. Soc. {\bf 62}, no. 348 (1986).

		

%\bibitem{Rosenberg-Schochet:UCT}
%Rosenberg, J., Schochet, C.
%\emph{The K\"unneth theorem and the universal coefficient theorem for 
%Kasparov's generalized \(\K\)-functor}. 
%Duke Math. J. (2), no. 2 (1987), pp. 431--474. 

\bibitem{Segal}
G. Segal. \emph{Equivariant $K$-theory}. Inst. Hautes \'Etudes Sci. Publ. Math. {\bf 34} (1968), 129--151.

\end{thebibliography}
\end{document}